\documentclass[a4paper,11pt]{amsart}
\usepackage[latin1]{inputenc}

\usepackage{multicol}
\usepackage{amssymb, amsmath, amsthm}
\usepackage{latexsym}
\usepackage{dsfont}
\usepackage{graphicx}

\usepackage{curves} 
\usepackage{epic} 
\usepackage{eepic}
\usepackage{xcolor}
 
\input xy
\xyoption{all}

\usepackage{verbatim}

\usepackage{hyperref}

\newtheorem{theorem}{Theorem}[section]

\newtheorem{corollary}[theorem]{Corollary}

\newtheorem{lemma}[theorem]{Lemma}

\newtheorem{proposition}[theorem]{Proposition}

	{\par\noindent{\bf Proposition \ref{res:hiper}.}\!\!
	\nopagebreak\normalsize\it}%

	{\par\noindent{\bf Theorem \ref{result43}.}\!\!
	\nopagebreak\normalsize\it}%

	{\par\noindent{\it Sketch of the proof}.  
	\nopagebreak\normalsize}%
	{\hfill\linebreak[2]\hspace*{\fill}$\circlearrowleft$}

	{\par\noindent{\it Proof of Proposition }\ref{prop:stab:smc}.  
	\nopagebreak\normalsize}%
	{\hfill\linebreak[2]\hspace*{\fill}$\circlearrowleft$}

	{\par\noindent{\it Proof of Propositions }\ref{adap:mon}{\it and }\ref{simult:adap}.\!\!\!
	\nopagebreak\normalsize}%
	{\hfill\linebreak[2]\hspace*{\fill}$\circlearrowleft$}

{\theoremstyle{definition}
       \newtheorem{definition}[theorem]{Definition}

       \newtheorem{remark}[theorem]{Remark}
       \newtheorem{example}[theorem]{Example}
       \newtheorem{parrafo}[theorem]{{\!}}  }

\numberwithin{equation}{theorem}

\newcommand{\calo}{{\mathcal {O}}}

\DeclareMathOperator{\Max}{\underline{Max}}
\DeclareMathOperator{\mult}{mult}
\DeclareMathOperator{\ord}{ord}

\DeclareMathOperator{\Sing}{Sing}
\DeclareMathOperator{\Spec}{Spec}

\DeclareMathOperator{\red}{red}

\newcommand{\G}{{\mathcal G}}

\newcommand{\p}{{\mathfrak p}}

\newcommand{\Z}{{\mathbb Z}}
\newcommand{\N}{{\mathbb N}}

\definecolor{darkpurple}{rgb}{0.28,0.24,0.55}
\definecolor{lightblue}{rgb}{0,0.75,1}

\title{Equimultiplicity, algebraic elimination, and blowing-up.}
\author{Orlando E. Villamayor U.}
\thanks{2000 {\em Mathematics subject classification. 14E15.}}
\thanks{The author is partially supported by MTM2009-07291.}
\date{June 2010}
\address{Dpto. Matem\'aticas,  Universidad
Aut\'onoma de Madrid and Instituto de Ciencias Matem\'aticas CSIC-UAM-UC3M-UCM \\
Ciudad Universitaria de Cantoblanco, 28049 Madrid, Spain}

\email[Orlando E. Villamayor U.]{villamayor@uam.es}
\keywords{Multiplicity, integral closure, Rees algebras, elimination, perfect fields.}
\begin{document}
\begin{abstract}
Given a variety $X$ over a perfect field, we study the partition defined on $X$ by the multiplicity (into equimultiple points), and the effect of blowing up at smooth equimultiple centers. Over fields of characteristic zero we prove resolution of singularities by using the multiplicity as an invariant, instead of the Hilbert Samuel function. 

%

\end{abstract}
\maketitle

{\tableofcontents}

%
%
%
%
%
%
%
%
\color{black}

\section{Introduction}

Here $X$ will be a scheme of finite type over a perfect field $k$, we shall study properties of the multiplicity along points of $X$ using tools of commutative algebra, such as integral closure of ideals.  We also assume, throughout this paper, that $X$ is equidimensional.
The multiplicity defines a function, say
\begin{equation}\label{lmul} \mult_X: X \to \mathbb{N},
\end{equation}
where the domain is the underlying topological space of $X$, and given $\xi \in X$, $\mult_X(\xi)$ is the multiplicity of the local ring $\calo_{X,\xi}$.  This function is  upper semi-continuous, as we shall indicate below, so the level sets are locally closed.  Here  $\xi$ is said to be an n-fold point when $\mult_X(\xi)=n$. Let $$F_n(X)$$ denote the set of points of $X$ of multiplicity $n$, or say the set of $n$-fold points. We prove that there is a {\em local presentation} at the point. Namely, locally at $\xi$, in the sense of \'etale topology, there is an embedding in a smooth scheme, say $X\subset W$, and hypersurafces
\begin{equation}\label{intro1}
\mathcal{H}_1, \dots , \mathcal{H}_r 
\end{equation}
and integers $ n_1, \dots , n_r$, where each $\mathcal{H}_i$ has multiplicity $n_i$ at $\xi$, and 
\begin{equation}\label{intro2} F_n(X)= \bigcap_{1\leq i \leq r} \Sing(\mathcal{H}_i, n_i)
\end{equation}
where $\Sing(\mathcal{H}_i, n_i)\subset W$ is the set of points where the hypersurface $\mathcal{H}_i$ has multiplicity $n_i$. The requirement on this local presentations is that such expression is preserved for the set of all infinitely near singularities having the same multiplicity $n$, namely: If $Y\subset 
F_n(X)= \bigcap_{1\leq i \leq r} \Sing(\mathcal{H}_i, n_i)$ is a regular center, and if $X\leftarrow X_1$ and $W\leftarrow W_1$ denote the blow ups at $Y$, then 
\begin{equation}\label{intro3}  F_n(X_1)= \bigcap_{1\leq i \leq r} \Sing(\mathcal{H}^{(1)}_i, n_i)\subset W_1
\end{equation}
where $\mathcal{H}^{(1)}_i$ is the strict transform of $\mathcal{H}_i$.

Moreover, we require this expression to hold after any sequence of blow-ups on smooth centers included in the set of $n$-fold points. 

This parallels Hironaka's notion of idealistic space (\cite{Hironaka77}) , which is a local presentation that he attaches to points with a given Hilbert-Samuel function, whereas here we assign a local presentation to the multiplicity.

The interest of local presentations is that they simplify the study of the behavior of these invariants when blowing up at regular center. In fact, they render a {\em reduction to the hypersurface case} as they allow us to replace $X$ by a finite number of hypersurfaces, and the law of transformation of a hypersurface is easy to handle. 

 In (\cite{Hironaka77}) Hironaka makes use of techniques of division at henselian rings to obtain the hypersurfaces of the local presentation corresponding to the Hilbert Samuel function (a result of Aroca). Here we use algebraic techniques, which are very close to those of algebraic elimination,
to construct the local presentation corresponding to the multiplicity.
%

Let us recall the notion of Hilbert-Samuel function mentioned above. There is firstly a Hilbert function attached to a  point of a noetherian scheme $\xi \in X$, say 
$HS_X(\xi)$, that maps $\mathbb{N} $ to $\mathbb{N}$. So the graph is in $\Lambda=\mathbb{N}^\mathbb{N}$ which we consider now with the lexicographic ordering. The Hilbert Samuel function $HS_X:X\to \Lambda$ is constructed by setting $HS_X(\xi)$ as before when $\xi \in X$  is a closed point, and with a prescribed modification on non-closed points.

Bennett proved that if $X$ is an excellent scheme, the Hilbert Samuel function $HS_X$ is upper semi-continuos.
Therefore the level sets of the function stratify the excellent scheme $X$ into locally closed sets. 
A second fundamental result, also due to Bennett, says that if $\alpha: X_1 \to X$ is the blow up at a closed and regular center included in a stratum, then $HS_{X_1}(y) \leq HS_{X}(\alpha(y))$ at any point $y\in X_1$ (see \cite{BB}).

This latter result was further simplified by Balwant Singh in 
\cite{BS}. We refer here to \cite{CJS}, where the semicontinuity of
$HS_X$ is studied in the first chapter (Theorem 1.34), and the behavior under blow ups is carefully treated in the second chapter.

This study of the Hilbert Samuel stratification took place in the early 70's, and led to the extension of Hironaka's theorem to the class of excellent schemes over a field of characteristic zero. 

The expression of the multiplicity of a local ring in terms of the Hilbert Samuel polynomial is due to Samuel (see \cite{Samuel}), and it follows from this that if $X$ is an excellent pure dimensional scheme (with the same dimension locally at any closed point), the stratification defined by $HS_X$ is a refinement of that of $\mult_X$. In particular, in this case $\mult_X$ is semi-continuous as a result of Bennett's work.
However, the concept of multiplicity is older then that of Hilbert functions, and the semi-continuity of $\mult_X$ was already studied in earlier years.

 A first step in this direction is a theorem of Nagata that states that if $p$ is a prime ideal in a local ring $R$, then the inequality $e(R)\geq e(R_p)$ holds if $p$ is analytically unramified (i.e., if the completion of the local ring $R/p$ is reduced), and 
 $\dim R=\dim R_p +\dim R/p$ (see \cite{Nagata1}, or Theorem 40.1 in \cite{Nagata}). 
If $X$ is an excellent scheme, with the additional property that all saturated chains of irreducible subschemes $Y_0 \subset Y_1 \dots \subset Y_d$ have the same length, then Nagata's criterion applies, hence $\mult_X(x)\geq \mult(y)$ when $x\in \overline{y}$ in $X$.
This condition is not sufficient to prove the semi-continuity, and 
a second step, due to Dade (\cite{D}), settles this property. He uses there ideas and invariants introduced by Northcott and Rees,
that we shall mention later.

Dade also proves that if $\alpha: X_1 \to X$ is the blow up at a closed and regular center included in a stratum (in a level set) of $\mult_X$, then $\mult_{X_1}(y) \leq \mult_{X}(\alpha(y))$ at any point $y\in X_1$.
 This latter result, was simplified and generalized by Orbanz \cite{Orbanz}).
 
%
%
%
%


In this paper we will consider equidimensional schemes 
of finite type over a perfect field.  In \ref{lem51c} we will prove these previous results in this restricted setting, without using the results of Bennett or Dade. We discuss their results in \ref{XXX1}, which hold for more general excellent  schemes.
%
\color{black}

The main objective in this paper is the discussion of local presentation. Let us recall here that these lead to resolution of singularities over fields of characteristic zero. 

\begin{theorem}(Hironaka)\label{Res}  Given a variety $X$ over a field of  characteristic zero, there is a sequence 
\begin{equation}\label{nccm1}X \longleftarrow X_1\longleftarrow \dots \longleftarrow X_s.
\end{equation}
of blow ups at regular centers so that:
\begin{itemize}
\item $X_s$ is regular.
\item The composition $X\leftarrow X_r$ induces an isomorphism on $X\setminus \Sing X$;
\item  The   exceptional divisor of $X\leftarrow X_r$ has normal crossing support.  

\end{itemize}
\end{theorem}
This is an existential theorem, whereas we will address here resolution from a constructive point of view, by fixing an algorithm of resolution of singularities. 

Recall that, in general, there are different resolutions for a given singular variety $X$. However, if we fix an algorithm one can require some natural properties: 
\begin{itemize}
\item[i)] It assigns to each $X$ a unique resolution of singularities.
\item[ii)]  If $X\leftarrow X'$ is smooth, the constructive resolution of $X'$ is the pull-back of that of $X$
\item[iii)] (Equivariance) If a group acts on $X$, then the group action can be lifted to the constructive resolution of $X$.
\end{itemize}
These properties are known to hold for constructive resolution in
which local presentations are defined in terms of the Hilbert-Samuel function (\cite{Villa92}). The same argument applies here, where 
local presentations will be attached to the multiplicity. The guideline in both cases is the notion of equivalence introduced by Hironaka, and used in \cite{Hironaka77}, where he indicates that the equivariance follows from a more general form of {\em compatibility with isomorphisms}:
\begin{itemize}
\item[iv)] Let $\Theta: X \to Y$ be an isomorphism of schemes. Assume that $X$ and $Y$ have structures of varieties over fields $k$ and $k'$ respectively, both of characteristic zero. Then the resolution of singularities assigned by the algorithm to $X$ is the pull-back, via 
the isomorphism, of that assigned to $Y$.
\end{itemize}
In this paper we construct the sequence (\ref{nccm1}) by blowing up equimultiple centers, with no reference to normal flatness. This answers Question D in \cite{Hironaka64}, p. 134, and renders a curious compatibility:

A variety $X$ over a field $k$ of  characteristic zero is in particular an equidimensional scheme of finite type over $k$. Let $Y$ be an equidimensional scheme of finite over $k$, and let $(Y)_{red}$ denote the reduced scheme. Our procedure introduces a new natural property, of compatibility with reductions. It will assign to each $Y$ a unique sequence 
$Y \longleftarrow Y_1\longleftarrow \dots \longleftarrow Y_s$.
In addition, if we set $X_i=(Y_i)_{red}$ we recover the constructive resolution of  $X=(Y)_{red}$.

The connection between the blow up at equimultiple centers, and at centers with the same Hilbert-Samuel function, has been studied in several works, particularly by Hironaka and by Schickhoff. We refer here to \cite{Herrmann} , or to section 5 in  \cite{Lipman1}, for a detailed report on that progress.

\vskip 0,5cm

 The previous discussion indicates that local presentation are relevant for constructive resolution, a point to be addressed in the last Section 8. 

 Note here that Theorem \ref{Res} is formulated in the class of schemes of finite type over a field $k$, and we want to carry out this discussion within this class. In fact we show that local presentations can be attached to the multiplicity within the class of schemes of finite type over a perfect field. We now sketch the general strategy and 
indicate why we need $k$ to be perfect.
 
Lipman studies the multiplicity for complex analytic varieties, and also in the algebraic setting. Let $Z$ be a complex analytic variety, and let $x\in Z$ be a point where the variety has multiplicity $n$ and local dimension $d$, then one can construct, at a suitable neighborhood, a finite morphism $\delta: Z\to (\mathbb{C}^d,0)$, with $n$ points in the general fiber. Moreover, $n$ is the smallest integer with this property. So
$Z$ can be expressed, locally, as a finite and ramified cover of $\mathbb{C}^d$. This is the starting point for the study of equimultiplicity in Proposition 4.1 of \cite{Lipman1}.

In this paper $X$ is an equidimensional scheme of finite type over a perfect field $k$. Fix a closed point $\xi \in X$ of multiplicity $n$. Under these hypothesis, and after restricting $X$ to an \'etale neighborhood of the point, one can construct a finite morphism $\delta: X\to V$, where $V$ is smooth over $k$, and there are $n$ points in the general fiber. To achieve this we first note that, as the base field $k$ is perfect and  $\xi \in X$ is closed, we may assume first that the point is rational. In fact this requires a change of the base field which is an \'etale morphism. One can then repeat the previous construction (in the analytic realm), and produce a finite morphism at the henselization of $\calo_{X,\xi}$. Finally, one can descend this result to a suitable \'etale neighborhood of $\xi \in X$ (see e.g., \cite{BrV1}, Appendix A)). The fact that $k$ is perfect and that $X$ is a scheme of finite type is crucial for this construction.
A finite morphism with this property is said to be {\em transversal} at the point $\xi$.

\color{black}
%
%
%

%
%

 {\bf 1)}  A fundamental property of a transversal morphism concerns the way it maps the set of $n$-fold points $F_n(X)$ into its image, say $\delta(F_n(X))$. The main features are:

{\bf 1,a)} The morphism $\delta: X \to V$, when restricted to $F_n(X)$, induces a homeomorphism 
\begin{equation}
F_n(X) \equiv \delta(F_n(X))
\end{equation}

 {\bf 1,b)} An irreducible  subscheme $Y\subset F_n(X)$ is smooth if and only if $\delta(Y)(\subset \delta(F_n(X)))$ is smooth.

{\bf 2)} The blow up of $X$ at a smooth center $Y\subset F_n(X)$ induces a commutative diagram 

\begin{equation}\label{di5w} \xymatrix{
 X\ar[d]_{\delta}& & X_1\ar[d]_{\delta_1} \ar[ll]\\
 V&   & \ar[ll]  V_1\\
 } 
\end{equation}
where the horizontal maps are the blow ups at $Y$ and $\delta(Y)$ respectively, and $\delta_1: X_1\to V_1$ is a finite morphism with $n$ points in the general fiber. We will show that points in $X_1$ have at most multiplicity $n$, so  if $F_n(X_1)\neq \emptyset$, then $\delta_1$ is transversal at any $n$-fold point.

The properties { 1,a)} and { 1,b)} indicate that the set of $n$-fold points in $X$ are evenly spread over their image by a transversal morphism, whereas { 2)} is a property of {\em stability} of the transversality. These two aspects will be justified in this paper, together with other relevant features of the multiplicity:

%

%
%
{\bf 3)}  The finite transversal morphisms are constructed locally in \'etale topology, and we may assume that $\delta: X \to V$ is affine. Namely, that $V=Spec(S)$, $X=Spec(B)$, where 
$S\subset B$ is a finite extension of algebras of finite type over the perfect $k$, and that $S$ is smooth over $k$ and irreducible. Set 
$B=S[\Theta_1, \dots, \Theta_r]$, which we view as a quotient of a polynomial ring $S[X_1, \dots, X_r]$, and hence there is a commutative diagram 
\begin{equation}\label{eccp}
\xymatrix{
  X=\Spec(B) \ar[d]_{\delta} &\subset  W=\Spec(S[X_1, \dots ,X_r])  \ar[dl]_{\delta'}& \\ 
   V=\Spec(S)        &  &       
}
\end{equation}
where $\delta': W \to V$ is a smooth morphism of smooth schemes. For each index $i=1,\dots, r$ we will specify an equation  of integral dependence $$(\Theta_i)^{n_i}+a_{i,1} (\Theta_i)^{n_i-1}+ \dots +a_{i,n_i}=0,$$ 
and hence a polynomial, say
\begin{equation}\label{18}f_{\Theta_i}(X)=X^{n_i}+a_{i,1} X^{n_i-1}+ \dots +a_{i,n_i}\in S[X].
\end{equation}
Set $f_{\Theta_i}(X_i)\in S[X_1, \dots, X_r]$,  $i=1,\dots, r$. So $f_{\Theta_i}(X_i)$ maps to zero in $B$ via the surjective morphism of $S$ algebras $S[X_1, \dots, X_r] \to B$, mapping $X_j$ to $\Theta_j$, $j=1, \dots, r$.

We prove that a local presentation for the multiplicity is obtained from these polynomials define, setting 
$$\mathcal{H}_i=\{f_{\Theta_i}(X_i)=0 \} \subset W,$$ the hypersurface in $W$ defined by $f_{\Theta_i}(X_i)$ for $i=1, \dots, r$. Namely, we prove that

\begin{equation}\label{tacion}F_n(X)= \bigcap_{1\leq i \leq r} \Sing(\mathcal{H}_i, n_i)
\end{equation}
where  $\Sing(\mathcal{H}_i, n_i)$ denotes the points where the hypersurface $\mathcal{H}_i$ has multiplicity $n_i=\text{ deg} (f_{\Theta_i}(X_i))$; that the same holds after blowing up at a smooth  center $Y$ included in the $n$-fold points (see \ref{intro3}), and after applying  any sequence of blow ups with centers included in the $n$-fold points.

We use techniques of elimination, applied here for a finite extension $S\subset B$ over a smooth algebra $S$. This form of elimination was introduced as a tool to study singularities in positive characteristic. 

In Section 2 we present the Rees algebras. These will be used to reformulate the notion of local presentations.
Some aspects of this form of elimination treated here will be discussed in Section 3,.

Section 4 is devoted to some general results about the multiplicity, blow ups, and integral closure of ideals. The first steps are given here on the study of the behavior of the multiplicity on $B$, once we fix a finite extensions $S\subset B$ as above.
Although our ultimate goal is to study the case in which both rings are finite type algebras over a perfect field, we also consider inclusions $S\subset B$ where the rings involved are more general. \color{black}

The property of transversality of finite morphisms, stated in {\bf 1a)}, is discussed in Section 4. Sections 5 and 6 are devoted to the properties {\bf 1)}, {\bf 2)}, and {\bf 3)}, these are the two main sections, both developed essentially in the local (affine) setting, whereas \ref{rk215}  contains some global results. 
In  \ref{lem51c} we discuss the theorem of Dade-Orbanz, and in Proposition \ref{lem512} we mention some well known properties of the multiplicity, relating the partition of $X$ obtained from the multiplicity (in equimultiple points) with that on the reduced scheme $X_{red}$.
The discussion in sections 5 and 6 renders our formulation of local presentation corresponding to the multiplicity, which is summarized in Section 7.

Finally, in Section 8 we recall briefly the role of local presentation 
in resolution of singularities when the characteristic is zero.

In the paper we study properties of the multiplicity using classical results of commutative algebra. We pursue on the line of Lipman's exposition (section 4 in \cite{Lipman2}) where he presents the property of transversality in terms of finite extensions of rings and integral closure of ideals. 
\color{black} The local rings of function that we consider here can be expressed, in many ways, as a finite extension of a regular ring. It is in this setting that one can apply techniques of ramification. This has been a classical framework to study the multiplicity, and this strategy appears already in other works, and led to different approaches, for instance in Cutkosky's presentation in \cite{Cut} (see also \cite{Ab1}, \cite{CP1}, \cite{D}, and \cite{kaw}).  

I thank A. Nobile for several suggestions on a previous version. I also thank the referee for numerous indications that improved 
this presentation.

%
%
%
%

\section{Rees algebras on smooth schemes.}

\begin{parrafo}\label{sec45}
%
%
Assume, for simplicity, that $X$ is a scheme that can be embedded as a hypersurface in a smooth scheme $V$ over a perfect field $k$. A hypersuface is defined by an invertible ideal, say $I(X)\subset \calo_V$, 
and the multiplicity of $X$ at a point $x$ is also the order of the principal ideal $I(X)$ at the local regular ring $\calo_{V,x}$.  

The strategy we follow here is to consider the closed subset of the hypersurface $X$ consisting of points with highest multiplicity. A {\em Rees Algebra} over the smooth scheme $V$ will be attached to this closed set. More precisely, we will also consider the closed set of points of highest multiplicity of those schemes obtained from $X$ after applying one, or even several blow ups at smooth centers. A fundamental observation of Hironaka, to be discussed here, is to show there is significant information on the multiplicity obtained by analyzing these closed sets. In fact this analysis leads to a stratification of the highest multiplicity locus, and to a notion of equivalence discuss here in \ref{Defweakx}. 

A regular center, say $Y$ in $X$, is also included in $V$, and the blow up of $X$ along $Y$ can be recovered from the blow up of $V$, say $V\leftarrow V_1$, by taking the {\em strict transform} of $X$, say $X_1\subset V_1$, which is also a hypersurface. This strategy will lead us to consider the blow up of $V$ along smooth centers over the underlying perfect field $k$. 
%
%
%
%
%

Fix ideals $\{I_n, n\in \N\}$ so that $I_0=\calo_V$, and $I_n I_m\subset I_{n+m}$. Consider the algebra
$${\mathcal G}=\oplus_nI_nW^n(\subset \calo_V[W]).$$
It inherits a graded structure given by the powers of the variable $W$. Or equivalently, it is a graded subalgebra of $ \calo_V[W]$.

We say that $\G$ is a Rees algebra over $V$ when it is locally a finitely generated algebra: When $V$ is restricted to affine open set, say $V'\subset V$, there are global sections $f_1, \dots, f_r$, and positive integers $n_1, \dots, n_r$, so that the restriction of $\G$ to such open set is of the form $\calo_{V'}[f_1W^{n_1}, \dots,f_rW^{n_r} ]$.

\end{parrafo}
\begin{parrafo} \label{singularlocus}{\bf }  \cite[1.2]{EV}
{\rm  Let $V$ be a smooth scheme over a field $k$,  and let
${\mathcal G}=\oplus_nI_nW^n$ be a Rees 
algebra over $V$.  Then  the  {\em singular locus of ${\mathcal
G}$}, is 
$$\mbox{Sing }{\mathcal G}:=\bigcap_{n\in {\mathbb N}_{>0}}\{x\in V: \nu_x(I_n)\geq n\},$$
where   
$\nu_x(I_n)$ denotes the  order of $I_n$ in the regular local
ring ${\mathcal O}_{V,x}$. Observe that $\mbox{Sing }{\mathcal G}$  is closed  in 
  $V$. Moreover, if ${\mathcal G}$ is generated by $f_1W^{n_1},\ldots,f_sW^{n_s}$, it is easy to check that
\begin{equation}\label{442}
\mbox{Sing }{\mathcal G}=\bigcap_{i=1}^s\{x\in V: \nu_x(f_i)\geq n_i\}. 
\end{equation} }
\end{parrafo}

\begin{example} \label{ExampleA} 
{\rm Let $X\subset V$ be a hypersurface, and let $b$ be a non-negative 
integer. Then the singular locus of the Rees algebra generated by 
$I(X)$ in degree $b$, say $\calo_V[I(X)W^b](\subset \calo_V[W])$,
  is the closed set of points of $X$ which have multiplicity at
least $b$ (which may be empty). In the same manner, if 
$J\subset {\mathcal O}_V$ is an arbitrary non-zero sheaf of ideals, and $b$ is a
non-negative integer, then the singular locus of 
 the Rees algebra
generated by $J$ in degree $b$, say $\calo_V[JW^b](\subset \calo_V[W])$, consists of
the points of $V$ where the order of $J$ is at least $b$.

Of major interest in our discussion will be the algebra $\G=\calo_V[I(X)W^b]$ where $b$ is the highest multiplicity of $X$. Then $\Sing\G$ is the closed set of points in $X$ of highest multiplicity.

}
\end{example}

\begin{parrafo} \label{orderRees}{\bf}
\cite[6.3]{EV} {\rm Let ${\mathcal
G}=\bigoplus_{n\geq 0}I_nW^n$ be a Rees algebra on a smooth scheme $V$,  let $x\in \mbox{Sing }{\mathcal
G}$, and let $fW^n\in I_nW^n$. Set
$$\mbox{ord}_x(fW^n)=\frac{\nu_x(f)}{n}\in {\mathbb Q}\geq 1,$$
where, as before, $\nu_x(f)$ denotes the order of $f$ in 
${\mathcal O}_{V,x}$. Note that $\mbox{ord}_x(fW^n)\geq 1$ since
$x\in \mbox{Sing }{\mathcal G}$. Let
$$\mbox{ord}_x{\mathcal
G}=\mbox{inf}\{\mbox{ord}_x(fW^n):fW^n\in I_nW^n, n\geq 1\}.$$ If
${\mathcal G}$ is generated by
$\{f_{1}W^{n_1},\ldots,f_{s}W^{n_s}\}$ then it can be shown that 
$$\mbox{ord}_x{\mathcal G}=\mbox{min}\{\mbox{ord}_x(f_{i}W^{n_i}):
i=1,\ldots,s\},$$ and therefore, since
 $x\in \mbox{Sing }{\mathcal
G}$, $\mbox{ord}_x{\mathcal G}$ is a rational number that is greater or equal to one. }

\end{parrafo}

\begin{parrafo} \label{weaktransforms} {\bf Transforms of Rees algebras by blow ups.}
A smooth closed subscheme $Y\subset V$ is said to be {\em permissible} for ${\mathcal G}=\oplus_{n}I_nW^n\subset{\mathcal O}_{V}[W]$ if     $Y\subset
\mbox{Sing } {\mathcal G}$.  A {\em permissible transformation}  is the blow up at a permissible center $Y$, say $V\leftarrow V_1$. If 
$H_1\subset V_1$ denotes  the exceptional divisor, then for each
$n\in {\mathbb N}$, there is a unique factorization, say
$$I_n{\mathcal O}_{V_1}=I(H_1)^nI_{n,1}$$
for some  sheaf of ideals $J_{n,1}\subset {\mathcal
O}_{V_1}$, called the {\em weighted transform} of $J_n$.  The   {\em  transform of ${\mathcal
G}$}  in $V_1$ is: 
$${\mathcal G}_1:=\oplus_{n}I_{n,1}W^n.$$

\end{parrafo}

The next proposition gives a local description of the 
transform of a Rees algebra.
\begin{proposition}\label{localt}\cite[Proposition 1.6]{EV}
 Let ${\mathcal G}=\oplus_nI_nW^n$ be a
Rees algebra over a scheme $V$, which is smooth over a field $k$, and let
$V\leftarrow V_1$ be a permissible transformation. Assume, for simplicity, that 
$V$ is affine. If  
${\mathcal G}$ is generated by
$\{f_{1}W^{n_1},\ldots,f_{s}W^{n_s}\}$, then its  transform ${\mathcal
G}_1$ is generated by
$\{f_{1,1}W^{n_1},\ldots,f_{s,1}W^{n_s}\}$,
where   $f_{i,1}$ denotes a weighted transform of $f_{i}$ in $V_1$ 
for  $i=1,\ldots,s$.
\end{proposition}
\proof Assume, for simplicity, that the center of the transformation 
is irreducible. Let $y\in V$ denote the generic point of this smooth center. Then $y\in \mbox{Sing }{\mathcal G}$, or equivalently 
$\nu_y(f_i)\geq n_i$, for $i=1, \dots ,s$. In particular, at any affine chart in $V_1$, there are expressions of the form 
$$f_i{\mathcal O}_{V_1}=I(H_1)^{n_i}f_{n_i,1}.$$
It is easy to verify that $\{f_{1,1}W^{n_1},\ldots,f_{s,1}W^{n_s}\}$ generate ${\mathcal
G}_1$ at such affine chart of $V_1$.
\endproof

If $\G=\calo_V[I(X)W^b]$ is the algebra attached to a hypersurface $X$, where $b$ is the highest multiplicity of $X$, then $\G_1$ is the algebra attached to $X_1$, the strict transform of $X$ in $V_1$, and $\Sing\G_1$ is the closed set of points in $X_1$ of multiplicity $b$.

\begin{parrafo}\label{pback} {\bf } It will be technically useful, in our discussion, to consider a new kind of transformation. Given a smooth scheme $V$ and a positive integer $s$, then 
$ V \stackrel{pr^{(s)}}{\longleftarrow}V\times_{k}\mathbb{A}^s$
will denote the projection on the first coordinate. A Rees algebra 
${\mathcal G}=\oplus_nI_nW^n$ on $V$ has a natural lifting to a Rees algebra on $V\times_{k}\mathbb{A}^s$, say 
$(pr^{(s)})^{*}(\mathcal {G})$, obtained simply by taking extensions. 
\end{parrafo}
\begin{definition} 
\label{local_sequenceV} 
{\rm Let $(V,\mathcal{G})$ denote a Rees algebra ${\mathcal G}$ over $V$. A sequence over $(V,\mathcal{G})$, say
\begin{equation}\label{svlap3}
\begin{array}{ccccccc} 
V^{}_0=V^{} & \leftarrow & V^{}_1 & \leftarrow & \ldots & \leftarrow V^{}_m\\
{\mathcal G}_0^{}= {\mathcal G}^{} & & {\mathcal G}_1^{} & &\ldots  & {\mathcal G}_m^{}
\end{array}
\end{equation} 

%
will be constructed by setting for each index $i$, $i=0,1,\ldots,m-1$, $V_i \longleftarrow V_{i+1}$ either as:

A) a blow up at a center $Y_i \subset \Sing({\mathcal G}_i)$
${\mathcal G}_i$, and ${\mathcal G}_{i+1}$ is the transform of ${\mathcal G}_{i}$ 
in the sense of  \ref{weaktransforms}), 

B) $V_{i+1}= V_i\times_{k}\mathbb{A}^s$, for some integer $s\geq1$, where the morphism is the projection
$ V_i \stackrel{pr}{\longleftarrow}V_i\times_{k}\mathbb{A}^s,$
and ${\mathcal G}_{i+1}= (pr^{(s)})^{(*)}(\mathcal {G}_i)$ 

C) an open restriction, in which case ${\mathcal G}_{i+1}$ is the restriction of ${\mathcal G}_{i}$.}
\end{definition}

Fix  $(V,\mathcal{G})$ as before, and note that a sequence (\ref{svlap3})
provides a collection of closed subsets 
\begin{equation}\label{ABCs}
\mbox{Sing }{\mathcal G}_0\subset V_0, \ \    \mbox{Sing }{\mathcal G}_1\subset V_1,\ldots, \ \  \mbox{Sing }{\mathcal G}_m\subset V_m
\end{equation}

In our discussion the Rees algebra $\mathcal{G}$, or say the pair $(V,\mathcal{G})$, provides a procedure to obtain closed sets in this previous sense. Namely, the pair $(V,\mathcal{G})$ defines closed sets for any sequence (\ref{svlap3}). This leads to the following definition.

\begin{definition} \label{Defweakx}  Fix two Rees algebras $\mathcal{G}_1$ and $\mathcal{G}_2$ on a smooth scheme $V$. The pairs $(V,\mathcal{G}_1)$ and $(V,\mathcal{G}_2)$  
  are said to be {\em weakly equivalent} if: 

(i) $\mbox{Sing }\mathcal{G}_1=\mbox{Sing
}\mathcal{G}_2$;  

(ii) Any sequence of transformations of $(V,\mathcal{G}_1)$ induces transformations of $(V,\mathcal{G}_2)$,
 and conversely.

(iii) Given any sequence of transformations of $(V,\mathcal{G}_i)$ , for $i=1$ or $i=2$, say
\begin{equation}\label{svlap33}
\begin{array}{ccccccc} 
V^{}_0=V^{} & \leftarrow & V^{}_1 & \leftarrow & \ldots & \leftarrow V^{}_m\\
{\mathcal G}_{i,0}^{}= {\mathcal G_i}^{} & & {\mathcal G}_{i,1}^{} & &\ldots  & {\mathcal G}_{i,m}^{}
\end{array}
\end{equation} 
there is an equality of closed sets,  
$$\mbox{Sing }(\mathcal{G}_{1,j})=\mbox{Sing }(\mathcal{G}_{2,j})$$
for $0\leq j \leq m$. 
\end{definition}
\begin{remark}\label{rk210}
So $(V,\mathcal{G}_1)$  and $(V,\mathcal{G}_2)$ are weakly equivalent when they define the same closed sets. We shall not distinguish equivalent pairs. Concerning the condition C) in Definition \ref{local_sequenceV} , of restrictions to open sets: suppose that  $(V,\mathcal{G})$ is a pair and fix two open sets $U$ and $U'$ in $V$. Note that if $U\cap \mbox{Sing }{\mathcal G}=U'\cap \mbox{Sing }{\mathcal G}$, then we will not distinguish the restriction on $U$ with that on $U'$.
The following fundamental result, due to Hironaka, illustrates the importance of weak equivalence.
\end{remark}
\begin{theorem}\label{Th416}{\bf } If $(V,\mathcal{G}_1)$ and $(V,\mathcal{G}_2)$ are {\em
 weakly equivalent}, 
 then:

1) $\mbox{Sing }{\mathcal G_1}=\mbox{Sing }{\mathcal G_2}$ in $V$.

2) Given $x\in \mbox{Sing }{\mathcal G_1}$, $\mbox{ord}_x{\mathcal
G_1}=\mbox{ord}_x{\mathcal
G_2}.$
\end{theorem}	

We simply indicate that the apparently artificial concept of multiplication by affine spaces, introduced in \ref{pback}, is crucial for the proof of part 2) (see \cite{Cut2}, Proposition 6.27). 
The original proof is a particular application of the notions of Groves and Polygroves developed by Hironaka in the seventies (see e.g., \cite{Hironaka771}, \cite{Hironaka777}). These notions lead to the concept of weak equivalence, and they also played a central role in the proof of the natural properties of constructive resolution, as we shall indicate along this paper.

\begin{parrafo}\label{GenRes1} {\em On local presentations I.} One of the main objectives of this paper is that of defining local presentations for the multiplicity. These will be formulated in terms of Rees algebras. 

We start with a variety $X$ over a perfect field $k$, and we fix a point $\xi\in X$ of multiplicity $n$. After a suitable restriction of $X$ to an \'etale neighborhood of  the point, say $X$ again, we will show that there is an embedding 
\begin{equation}\label{tacionnn} X\subset W,
\end{equation} where $W$ is affine and smooth over $k$, together with a Rees algebra $\mathcal{G}$ on $W$ so that 
\begin{equation}\label{tacionn}F_n(X)=\Sing(\mathcal{G})
\end{equation}
where $n$ is the highest multiplicity and $F_n(X)$ denotes the set of $n$-fold points of $X$.

If $W_{1}= W\times_{k}\mathbb{A}^s$, for some integer $s\geq1$, 
then there is a natural inclusion 
$$X_1=X \times_{k}\mathbb{A}^s\subset W_1= W\times_{k}\mathbb{A}^s.$$

Note that the $F_n(X_1)=pr^{(s)})^{(-1)}(F_n(X))$, and $F_n(X_1)=\Sing( (pr^{(s)})^{(*)}(\mathcal {G}))$
where  
$ W_1 \stackrel{pr}{\longrightarrow} W$
denotes the projection, and  $(pr^{(s)})^{(*)}(\mathcal {G})$ is obtained by extending the algebra $\mathcal {G}$ to an algebra over $ W_1$ (total transform).
Moreover, we will produce $\mathcal{G}$ so that any sequence 

\begin{equation}\label{svlapf3}
\begin{array}{ccccccc} 
W^{}_0=W^{} & \leftarrow & W^{}_1 & \leftarrow & \ldots & \leftarrow W^{}_m\\
{\mathcal G}_0^{}= {\mathcal G}^{} & & {\mathcal G}_1^{} & &\ldots  & {\mathcal G}_m^{}
\end{array}
\end{equation} 
as that in (\ref{svlap3}), induces a sequence over $X$, say 

\begin{equation}\label{svlapf4}
\begin{array}{ccccccc} 
W^{}_0=W^{} & \leftarrow & W^{}_1 & \leftarrow & \ldots & \leftarrow W^{}_m\\
\cup & & \cup & & & \cup \\
{X}_0^{}= X & & {X}_1^{} & &\ldots  & {X}_m^{}
\end{array}
\end{equation} 
and, in addition 
\begin{equation}\label{xdr}\Sing( {\mathcal G}_i^{} )=F_n(X_i), \  \ i=0,1, \dots, m.
\end{equation}

Some indications will be given below as to how these algebras relate to local presentations, at least as this latter notion appears in the introduction. Before doing so let us formulate the definition of resolution a Rees algebra, and indicate why this would lead to the simplification of singularities over perfect fields. 
\end{parrafo}

\begin{parrafo}\label{transp}{\em On Rees algebras and hypersurfaces with normal crossings.}

Let $V$ be smooth over a field $k$, and let 
$E=\{H_1, \dots , H_s\}$ be a collection of smooth hypersurfaces having only normal crossings. This condition arises in the formulation of resolution of singularities. 

A blow up with a smooth center $Y$ is said to be {\em permissible} for $(V, E)$ if $Y$ has normal crossings with the union of hypersurfaces in $E$. In this case let $V\leftarrow V_1$ be the blow up, and let
$$(V,E=\{H_1, \dots , H_s\}) \leftarrow (V_1, E_1=\{H_1, \dots , H_s, H_{s+1}\}),$$
denote the {\em transform} of $(V,E)$, where the hypersurface $H_i\in E_1$ is the strict transform of $H_i\in E$, for $i=1, \dots , s$, and $H_{s+1}$ is the exceptional hypersurface introduced by the blow up.

There are several problems of {\em embedded resolution} where special attention is to be drawn on the hypersurface that arise as exceptional hypersurfaces, after applying a sequence of blow ups. One is that of embedded resolution of singularities, in which we start with a singular scheme $X$ embedded in a smooth scheme $V$ and we want to obtain a sequence of blow ups over $V$ so that all exceptional hypersurfaces have normal crossings, and the strict transform of $X$ is smooth and also has normal crossings with the exceptional hypersurfaces.

Another related result is that of {\em log-principalization} of ideals. 
There we start with an ideal, say $I$ in $V$, and we want to obtain 
a sequence of blow-ups as above, so that the total transform of $I$  is an invertible sheaf of ideals supported on smooth hypersurfaces having only normal crossings.
\begin{definition}
\label{local_sequenceA} 
{\rm  Let $V_0$ be smooth over field $k$, and let $E_0$ be a collection of smooth hypersurfaces with only normal crossings. 
A {\em sequence over   $(V_0,E_0)$}  is a sequence of the form 
$$(V_0,E_0) \longleftarrow
 (V_1,E_1)  \longleftarrow \cdots \longleftarrow
(V_m,E_m)$$
where  for $i=0,1,\ldots,m-1$, each $(V_i,E_1) \longleftarrow 
(V_{i+1},E_{i+1})$ is either an open restriction, 
the blow up at a smooth closed subscheme as in \ref{sec45}, or $V_{i+1}= V_i\times_{k}\mathbb{A}^s$, for some integer $s\geq 1$, the morphism is 
$ V_i \stackrel{pr}{\longleftarrow}V\times_{k}\mathbb{A}^s,$
and $E_{i+1}$ is the collection obtained by taking the pull-back of hypersurfaces in $E_i$. }
\end{definition}

\end{parrafo}

\begin{definition} 
\label{local_sequenceB} 
{\rm Let $(V,E)$ be as above, and let ${\mathcal G}$ be a Rees algebra over $V$. We denote these data by $(V,\mathcal{G},E)$, called {\em basic object}. A sequence over $(V,\mathcal{G},E)$, say
\begin{equation}\label{ABC1}
(V,\mathcal{G},E)=(V_0, \mathcal{G}_0,E_0) \longleftarrow
(V_1,\mathcal{G}_1, E_1)\longleftarrow \cdots \longleftarrow
(V_m,\mathcal{G}_m, E_m),
\end{equation}
will be constructed by setting for index $i$, $i=0,1,\ldots,m-1$, $V_i \longleftarrow V_{i+1}$ either as:

A) a blow up at a center $Y_i$, $(V_i, E_i) \leftarrow (V_{i+1}, E_{i+1})$ as in \ref{sec45}, and, in addition $Y_i$ is permissible  for 
${\mathcal G}_i\subset {\mathcal O}_{V_i}[W]$ (and ${\mathcal G}_{i+1}$ is the transform of ${\mathcal G}_{i}$ 
in the sense of  \ref{weaktransforms}), 

B) $V_{i+1}= V_i\times_{k}\mathbb{A}^s$, for some integer $s\geq1$, the morphism is the projection
$ V_i \stackrel{pr}{\longleftarrow}V_i\times_{k}\mathbb{A}^s,$
${\mathcal G}_{i+1}= (pr^{(s)})^{(*)}(\mathcal {G}_i)$ and also 
$(V_i, E_i)\leftarrow  (V_{i+1}, E_{i+1})$ are constructed by taking pull-backs, or

C) an open restriction.}
\end{definition}

\begin{definition} \label{Defweak}  Fix $(V, E)$ as above. Two basic objects, $(V,\mathcal{G}_1,E)$ and $(V,\mathcal{G}_2,E)$  
  are said to be {\em equivalent} if: 

(i) $\mbox{Sing }\mathcal{G}_1=\mbox{Sing
}\mathcal{G}_2$;  

(ii) Any sequence of transformations of $(V,\mathcal{G}_1,E)$ induces transformations of $(V,\mathcal{G}_2,E)$,
 and conversely.

(iii) Given any sequence of transformations of $(V,\mathcal{G}_i,E)$ , for $i=1$ or $i=2$, say
\begin{equation}\label{bvc}
(V,\mathcal{G}_i,E)=(V_0,\mathcal{G}_{i,0}, E_0) \longleftarrow
(V_1,\mathcal{G}_{i,1}, E_1)\longleftarrow \cdots \longleftarrow
(V_m,\mathcal{G}_{i,m}, E_m),
\end{equation}
there is an equality of closed sets,  
$\mbox{Sing }(\mathcal{G}_{1,j})=\mbox{Sing }(\mathcal{G}_{2,j})$
for $0\leq j \leq m$. 
\end{definition}
\begin{remark} The conditions imposed on the sequences in Definition  \ref{local_sequenceB} seem to be more restrictive then those in Definition \ref{local_sequenceV}. However, 
two basic objects, $B=(V,\mathcal{G},E)$ and $B'=(V,\mathcal{K},E)$ are equivalent if and only if $\mathcal{G}$ and $\mathcal{K}$ are weakly equivalent (see \cite{Indiana}, Section 8).

\end{remark}

\begin{definition}\label{resbo1} Fix a basic object $\mathcal B_0=(V_0,\mathcal G_0, E_0)$, we say that a sequence 
\begin{equation}\label{ABC2}
\begin{array}{ccccccc}

(V_{0}, \mathcal G_0,E_0) & \longleftarrow & (V_{1}, \mathcal G_1,E_1) & \longleftarrow &
\cdots  &\longleftarrow &
(V_{m}, \mathcal G_m, E_m)
\end{array}
\end{equation}
as that in Definition  \ref{local_sequenceB}, is a {\em resolution} of $\mathcal B_0$ if it consists only on blow ups, and
$\Sing(\mathcal G_m)=\emptyset$.
\end{definition}

\begin{remark}\label{541} 1) If we fix an ideal $I$ on a smooth $V$, then the Rees ring of $I$, say ${\mathcal G}_I:=\oplus_{n}I^{n}W^n$ is a Rees algebra, and 
a resolution of $(V, {\mathcal G}_I,E=\emptyset )$ induces a log principalization of $I$ over $V$.

2) Fix $X\subset W$, an integer $n$, and an algebra $\mathcal{G}$ as in \ref{GenRes1}. Note that a resolution of the basic object 
$(W, {\mathcal G}, E=\emptyset )$ produces a sequences (\ref{svlapf3}) and  (\ref{svlapf4}) so that $F_n(X_m)=\emptyset.$
Namely $X_m$ has highest multiplicity strictly smaller then $n$.
Recall that a variety is regular when the highest multiplicity is one.

\end{remark}
\begin{parrafo}\label{GenRes2}{\em On local presentations II.} 
In  \ref{GenRes1} we fixed an embedding $X\subset W$, where $W$ is smooth over a perfect field $k$, and we claim that there is an algebra $\mathcal{G}$, attached to the $n$-fold points in $X$, with prescribed properties. 
The construction of $\mathcal{G}$ will be done locally, in \'etale topology. In fact, given $X$ we will construct, \'etale locally at any closed point an inclusion and a morphism
\begin{equation}\label{ledv} X\subset W \to V
\end{equation}
where $\beta: W\to V$ is a smooth morphism of smooth schemes inducing a finite morphism $\beta: X\to V$ (see 
(\ref{eccp})). Finally $\mathcal{G}$ will be an algebra over $W$ that will be constructed using these data. A property is that $F_n(X)=\Sing(\mathcal{G}) (\subset W)$ and the morphism will map the closed set $F_n(X)$ homeomorphically into its image $\beta(F_n(X))$ in $V$.
Moreover, for any sequence 
\begin{equation}\label{svlapf31}
\begin{array}{ccccccc} 
W^{}_0=W^{}  & \leftarrow & W^{}_1 & \leftarrow & \ldots & \leftarrow W^{}_m\\ 
{\mathcal G}_0^{}= {\mathcal G}^{} & & {\mathcal G}_1^{} & &\ldots  & {\mathcal G}_m^{}
\end{array}
\end{equation} 
obtained by blowing up at smooth permissible centers as in \ref{weaktransforms}, there will be a commutative diagram
\begin{equation}\label{svlapf32}
\begin{array}{ccccccc} 
W^{}_0=W^{}  & \leftarrow & W^{}_1 & \leftarrow & \ldots & \leftarrow W^{}_m\\ 
\downarrow & & \downarrow&&& \downarrow \\
{V}_0^{}= {V}^{} &\leftarrow  & {V}_1^{} & \leftarrow &\ldots  &\leftarrow  {V}_m^{}
\end{array}
\end{equation} 
where the lower row is a sequence of blow ups at regular centers, and where each $$\beta_i: W_i\to V_i$$ is smooth and maps the closed set $\Sing(\mathcal{G}_i)$ homeomorphically into its image 
$\beta_i(\Sing(\mathcal{G}_i))$ in $V_i$, for $i=0, \dots ,m$.
In addition, one obtains from (\ref{svlapf31})  

\begin{equation}\label{svlavc4}
\begin{array}{ccccccc} 
W^{}_0=W^{} & \leftarrow & W^{}_1 & \leftarrow & \ldots & \leftarrow W^{}_m\\
\cup & & \cup & & & \cup \\
{X}_0^{}= X &  \leftarrow& {X}_1^{} & \leftarrow &\ldots  & \leftarrow {X}_m^{}
\end{array}
\end{equation} 
where the lower row is a sequence of blow ups, and 
\begin{equation}\label{xdr}\Sing( {\mathcal G}_i^{} )=F_n(X_i), \  \ i=0,1, \dots, m.
\end{equation}

Recall that the map $X\to V$, obtained by the restriction to $X$ of $\beta: W\to V$, is finite. The same holds for $X_i\to V_i$, obtained by restriction of $\beta_i: W_i\to V_i$. 

The finite morphism $X\to V$ will be given by a finite 
extension $S \subset B=S[\Theta_1, \dots, \Theta_r]
$.
for a smooth $k$-algebra $S$, and $V=\Spec(S)$. The inclusion $X=\Spec(B) \subset W$ will be given by the surjection 
$T=S[X_1, \dots, X_r]\to S[\Theta_1, \dots, \Theta_r]$. So $W=\Spec(T)$, and we will assign a monic polynomial $f_{\Theta_i}(X_i)\in T$ to each 
$\Theta_i$. Finally we will set the Rees algebra
\begin{equation}\label{ledv1} \mathcal G=T[f_{\Theta_1}(X_1)Y^{n_1}, \dots, f_{\Theta_r}(X_r)Y^{n_r}] \subset T[Y].
\end{equation} 
where $Y$ is a variable over the ring $T$.

Note that the expression of the local presentation in (\ref{tacion}) follows from (\ref{442}) (see Prop. \ref{localt}). The arguments that will lead to the construction 
of the polynomials $f_{\Theta_i}(X_i)\in T$, with the previous properties, to be discussed in Sections 5 and 6, will be motivated by the notion of elimination.

\begin{remark}{\em (Over fields of characteristic zero)}

Resolution of basic object would lead to resolution of singularities. The existence of resolution of basic objects is known to hold over fields of characteristic zero. The proof of this result is addressed by induction on the dimension of the ambient space $V$.

Hironaka's approach for this form of induction is by choosing suitable smooth hypersurfaces in $V$, known as hypersurfaces of maximal contact, and replacing the basic object over $V$ by a basic object 
over this smooth hypersurface. 
There is alternative approach, in which the restriction of $V$ to smooth hypersurfaces is replaced by a smooth morphism, say $V\to V'$, where $V'$ is  smooth. This alternative approach involves a techniques of elimination. In fact one can construct an algorithm of resolution of a basic object using this technique. 

We refer to \cite{BrV1} to show how elimination leads to an algorithm of resolution of basic objects. Our discussion of resolution of singularities in the last section 8 will make use of 
this algorithm.
%
%
Let us mention, in passing, another application of this technique, that will follow from our discussion in \ref{Cor55}, valid if $k$ is a field of characteristic zero: One can 
attach to $V$ in  \ref{GenRes2} a new basic object, say $\mathcal{F}$, with the following property.
For all sequence (\ref{svlapf31}), the lower row in the sequence (\ref{svlapf32}) induces a sequence 
\begin{equation}\label{svlapf33}
\begin{array}{ccccccc} 
V^{}_0=V^{}  & \leftarrow & V^{}_1 & \leftarrow & \ldots & \leftarrow V^{}_m\\ 
{\mathcal F}_0^{}= {\mathcal F}^{} & & {\mathcal F}_1^{} & &\ldots  & {\mathcal F}_m^{}
\end{array}
\end{equation} 
and 
\begin{equation}\label{svlapf34}
\beta_i(\Sing({\mathcal G}_i))=\Sing({\mathcal F}_i) \ \ i=0, \dots, m.
\end{equation}
\end{remark}

\end{parrafo}


\begin{section}{Elimination algebras.}

We discuss here some techniques of algebraic elimination mentioned above, and relevant for the resolution of basic objects. The main result is Theorem \ref{eliRM}, formulated over fields of characteristic zero, where resolution of basic objects can be constructed. 

The aim of this paper is the study the multiplicity of varieties over perfect fields of arbitrary characteristic, and this section has also been included here because it largely motivates the discussion of the main results in this paper, addressed in Sections 6 and 7.

\color{black}
\begin{parrafo}\label{pel} {\bf Elimination algebras.} {\rm 
Assume that $X$ is a hypersurface in a smooth scheme $V^{(d)}$, of dimension $d$, and let $n$ denote the highest multiplicity at points of $X$. At a suitable \'etale neighborhood of a point $x\in \Sing(I(X),n)$ (at a point of multiplicity $n$), say $(V_1^{(d)}, x')$, there is a smooth morphism $\pi:V_1^{(d)}\to V^{(d-1)}$, and an element $Z$ of order one at $\calo_{V_1^{(d)},x'}$ so that: 
\begin{enumerate}
\item[(i)]  The smooth line $\pi^{-1} (\pi(x'))$ and the smooth hypersurface $\{ Z=0\}$ cut transversally at  $x'$.

\item[(ii)]  If $X'$ is the pull back of $X$ in $V_1^{(d)}$, then 
 $I(X')$ is spanned by a polynomial of the form $$f(Z)=Z^n+a_1Z^{n-1}+\ldots+a_1Z+a_0\in 
\calo_{V^{(d-1)}}[Z],$$
and the restriction of $V_1^{(d)}\to V^{(d-1)}$, namely $X'\to V^{(d-1)}$, is a finite morphism.
\end{enumerate}
(see \cite{BrV1}, Prop 32.3.)
%
%
This justifies the interest in studying the points of multiplicity $n$ of a hypersurface defined by a monic polynomial of degree $n$ (same integer $n$).
}
\end{parrafo} 
\begin{parrafo}\label{p002}
A first motivation of our forthcoming discussion can be formulated for polynomials over a field. Fix a field $K$ and a monic polynomial 
$f(Z)=Z^n+a_1Z^{n-1}+\ldots+a_{n-1}Z+a_n\in K[Z].$
If $K_1$ is a decomposition field of $f(Z)$, then 
$$f(Z)=(Z-\theta_1)\cdots (Z-\theta_n)\in K_1[Z],$$
and the coefficients $a_i (\in K)$ can  be expressed in terms of the elements $\{ \theta_1, \dots, \theta_n \}$ in $K_1$. In fact each coefficient $a_i$ is obtained from a symmetric function in $n$ variables, evaluated in  $(\theta_1, \dots, \theta_n )$. So, at least formally, and although the statement is not precise, one can set
\begin{equation}\label{ee1}
\Z[a_1, \dots , a_n]=\Z[\theta_1, \dots , \theta_n]^{S_n},
\end{equation}
where $S_n$ denotes the permutation group (see \ref{par01}). 
Consider a change of variable in $K[Z]$ obtained by fixing an element $\lambda \in K$ and setting $Z_1=Z-\lambda$. So 
$f(Z)=g(Z_1)$ in $K[Z]=K[Z_1]$. Say 
$$f(Z)=Z^n+a_1Z^{n-1}+\ldots+a_{n-1}Z+a_n=Z_1^n+b_1Z_1^{n-1}+\ldots+b_{n-1}Z_1+b_n=g(Z_1).$$
Fix the decomposition field $K_1$, as before, then $K_1[Z]=K_1[Z_1]$, and 

$$f(Z)=(Z-\theta_1)\cdots (Z-\theta_n)=(Z_1-\beta_1)\cdots (Z_1-\beta_n)=g(Z_1)$$
where $\beta_i=\theta_i+\lambda$, $i=1, \dots, n$.

Each coefficients $b_i$ is obtained by evaluation of a symmetric functions in $(\beta_1, \dots, \beta_n)$.
Our goal is to obtain polynomial expressions on the coefficient, say $H(V_1, \dots, V_n)\in \Z[V_1, \dots, V_n]$, so that 
\begin{equation}\label{ledi}
H(a_1, \dots, a_n)=H(b_1, \dots, b_n)
\end{equation}
every time when $g(Z_1)$ is $f(Z)$ expressed in a variable of the form $Z_1=Z-\lambda$, for some choice of $\lambda\in K$. A first observation is that for any such change of variable we get 
$$ \theta_i-\theta_j=\beta_i-\beta_j, \ 1\leq i, j\leq n.$$
Note that 
$\Z[\theta_i-\theta_j]_{1\leq i,j\leq
n}
\subset \Z[\theta_1, \dots , \theta_n]$
in particular, (\ref{ee1}) says that 
\begin{equation}\label{pri}
\Z[\theta_i-\theta_j]_{1\leq i,j\leq
n}
^{S_n}\subset \Z[a_1, \dots , a_n].
\end{equation}
Therefore, any element in the left hand side provides a polynomial expression in the coefficients, say 
$H(a_1, \dots, a_n)$, so that 
$H(a_1, \dots, a_n)=H(b_1, \dots, b_n)$
if $g(Z_1)=f(Z)$ for a change of variable.

\end{parrafo}
\begin{parrafo}\label{par01} Fix a monic  polynomial, say 
$f(Z)=Z^n+a_1Z^{n-1}+\ldots+a_1Z+a_0\in S[Z]$, over a smooth $k$-algebra $S$. 
 Assume here that $S$ is a domain, with quotient field $K$. As $S$ is smooth it is a normal ring. Considered $B=S[Z]/\langle f(Z)\rangle$ and the finite morphism
$\delta:\Spec (B)\to \Spec (S).$  It is natural to expect that there be significant information concerning this morphism, or say of $S\subset B$, which is encoded in the coefficients of $f(Z)$. Such is the case with the discriminant. 
On the other hand, if we let $Z_1=Z-\lambda$ for some $\lambda\in S$, then  $f(Z)=g(Z_1)\in S[Z]=S[Z_1]$ for some 
$g(Z_1)=Z_1^n+b_1Z_1^{n-1}+\ldots+b_{n-1}Z_1+b_n$, and
$B=S[Z]/\langle f(Z)\rangle=S[Z_1]/\langle g(Z_1)\rangle$.

In particular, if there is information of $\delta:\Spec (B)\to \Spec (S)$ in the coefficients of the polynomial it is reasonable to expect that it will not distinguish coefficients of $f(Z)$ from those of $g(Z_1)$.

To clarify these claims we bring the problem to a {\em universal context}. Fix a field $k$ and consider the polynomial ring in $n$ variables
$k[Y_1,\ldots,Y_n]$. The {\em universal polynomial} of degree $n$, is
$$F_n(Z)=(Z-Y_1)\cdots(Z-Y_n)=Z^n-s_{n,1}Z^{n-1}+\ldots+(-1)^ns_{n,n}\in
k[Y_1,\ldots,Y_n,Z],$$ where for $i=1,\ldots,n$, $s_{n,i}\in
k[Y_1,\ldots,Y_n]$ is the $i$-th symmetric polynomial.
The diagram
\begin{equation}\label{uncial}\xymatrix{
\operatorname{Spec}(k[s_{n,1}, \ldots, s_{n,n}][Z]/\langle F_n(Z)\rangle)  \ar@{^(->}[r] \ar@{^(->}[rd]_{\overline{\alpha}} & \operatorname{Spec}(k[s_{n,1}, \ldots, s_{n,n}][Z]) \ar[d]_{\alpha} \\
& \operatorname{Spec}(k[s_{n,1}, \ldots, s_{n,n}])
}
\end{equation}
illustrates the universal situation. In fact  $\delta:\Spec (B)\to \Spec (S)$ 
 is obtained from the specialization:
\begin{equation}
\label{especial}
\begin{array}{rrcl}
\Theta: & k[s_{n,1},\ldots,s_{n,n}] & \longrightarrow & S \\
 & (-1)^is_{n,i} & \to & a_i.
 \end{array}
 \end{equation}
In other words, there is a commutative diagram
\begin{equation}\label{di5w} \xymatrix{
 \mbox{Spec} (k[s_{n,1}, \ldots,s_{n,n}][Z]/\langle F_n(Z)\rangle) \ar[d]_{\overline{\alpha}}& & \Spec(B)\ar[d]_{\delta} \ar[ll]\\
 \mbox{Spec} (k[s_{n,1}, \ldots,s_{n,n}])&   & \ar[ll]\Spec(S)\\
 } 
\end{equation}
which, in addition, is a fiber product. Here 
\begin{equation}\label{eqdv}
k[Y_1,\ldots,Y_n]^{S_n}=k[s_{n,1},\ldots,s_{n,n}]
\end{equation}
is a polynomial ring, in particular it is smooth over $k$, and 
$\Theta: k[s_{n,1},\ldots,s_{n,n}]  \to  S $ is a morphism of $k$-algebras. Consider the subring $k[Y_i-Y_j]_{1\leq i,j\leq
n}\subset k[Y_1, \dots , Y_n]$
and note that the permutation group $S_n$ also acts on this subring. So there is an inclusion 
\begin{equation}\label{lidi}k[Y_i-Y_j]_{1\leq i,j\leq
n}^{S_n}\subset k[Y_1, \dots , Y_n]^{S_n}=k[s_{n,1},\ldots,s_{n,n}].
\end{equation}

As $S_n$ is a finite group the algebra in the left hand side is finitely generated. Set
\begin{equation} \label{unrs}
 k[G_{m_1},\ldots,G_{m_r}] :=k[Y_i-Y_j]_{1\leq i,j\leq
n}^{S_n}, 
\end{equation}
and, since $S_n$ acts linearly in $k[Y_1, \dots ,Y_n]$ (preserving the degree of this graded ring), we can take each generator
$G_{m_i}$ as an homogeneous polynomial in $k[Y_1, \dots, Y_n]$.
Let 
\begin{equation}\label{grdh}m_i=\mbox{degree }G_{m_i}
\end{equation} 
where $k[Y_1, \dots ,Y_n]$ is graded in the usual way. The inclusion (\ref{lidi}) yields an expression
\begin{equation}\label{eqdg}
G_{m_i}=G_{m_i}(s_{n,1},\ldots,s_{n,n})
\end{equation}
where $G_{m_i}(s_{n,1},\ldots,s_{n,n})$ is weighted homogeneous of degree $m_i$ in $k[s_{n,1},\ldots,s_{n,n}](\subset k[Y_1, \dots ,Y_n])$. 

This latter assertion says that there is a polynomial, say $G_{m_i}(V_1,\ldots,V_n)$, which is homogeneous of degree $m_i$ in $\Z[V_1, \dots, V_n]$, when this ring is graded so that each $V_i$ is given weight $i$. Here $G_{m_i}(s_{n,1},\ldots,s_{n,n})$ is obtained by setting $V_i=s_{n,i}$, $i=1, \dots, n$.

The morphism 
$\Theta: k[s_{n,1},\ldots,s_{n,n}]  \to  S $ maps 
$G_{m_i}(s_{n,1},\ldots,s_{n,n})$ to the element 
$G_{m_i}(a_1,\ldots,a_n)\in S$. Fix $\lambda\in S$ and set $S[Z]=S[Z_1]$ where $Z_1=Z-\lambda$. Let $$f(Z)=Z^n+a_1Z^{n-1}+\ldots+a_{n-1}Z+a_n=g(Z_1)=Z_1^n+b_1Z_1^{n-1}+\ldots+b_{n-1}Z_1+b_n.$$

Note finally that if $K$ denotes the quotient field of the domain $S$, and if 
$K_1$ is decomposition field of $f(Z)\in S[Z]\subset K[Z],$
then 
\begin{equation}\label{lapro}G_{m_i}(a_1,\ldots,a_n)=G_{m_i}(b_1,\ldots,b_n) \mbox{ in } S  \ (\mbox{see } ( \ref{ledi})).
\end{equation}
\begin{remark}\label{ark} Let $k[F_1, \dots, F_s]$ be a graded ring generated by homogeneous elements $F_i$, and set 
$m_i=\deg(F_i)$, $i=1, \dots, s$. Let $(R,M)$ be a local regular ring and a $k$-algebra, and let $\Theta: k[F_1, \dots, F_s] \to R$ be an homomorphisms of $k$-algebras. 
If  $ \Theta(F_i)$ has order $\geq m_i$, for $i=1, \dots , s$, then, for any homogeneous element $G\in k[F_1, \dots, F_s] $,
$ \Theta(G)$ has order $\geq d$ at $R$, where $d$ denotes the degree of $G$.
\end{remark}

\end{parrafo}

\begin{theorem}\cite[Theorem 1.16]{hpositive}
\label{eliRM} Let $k$ be a field of characteristic zero, and let $S$ be a smooth $k$-algebra. Fix
$f(Z)=Z^n+a_1Z^{n-1}+\ldots+a_{n-1}Z+a_n\in S[Z]$, and
\begin{equation}\label{casoesp1}\xymatrix{
\operatorname{Spec}(S[Z]/\langle f(Z)\rangle)  \ar@{^(->}[r] \ar@{^(->}[rd]_{\delta} & \operatorname{Spec}(S[Z]) \ar[d]_{\beta} \\
& \operatorname{Spec}(S).
}
\end{equation}

Let $F_n$ denote
the set of $n$-fold points of $\{f(Z)=0\}\subset
\Spec(S[Z])$.  Consider the  morphism obtained by specialization, say $ \Theta: k[s_{n,1},\ldots,s_{n,n}] \longrightarrow S$, 
$s_{n,i}  \to  (-1)^ia_i$. Then:
\begin{equation}\label{incSing}
\delta(F_n)= \bigcap_{1\leq j\leq r}\{x\in \Spec(S):
 \nu_x (G_{m_j}(a_1,\ldots,a_n)) 
\geq m_j\}, \end{equation} 
for $G_{m_j}$ as in (\ref{eqdg}) and $m_j$ as in (\ref{grdh}).
\end{theorem}
\proof Recall the description of the universal polynomial in 
\ref{par01} where the coefficients are the generators of $k[Y_1, \dots, Y_n]^{S_n}$ (see (\ref{eqdv})).
If the characteristic is zero (or if the characteristic does not divide $n$), one can check that
$$k[Y_1, \dots , Y_n]=(k[Y_i-Y_j]_{1\leq i,j\leq
n})[s_{n,1}],$$
where $s_{n,1}=Y_1+Y_2+\cdots +Y_n$. As $s_{n,1}$ is an invariant by the action of $S_n$, we conclude that 
$$k[Y_1, \dots , Y_n]^{S_n}=(k[Y_i-Y_j]^{S_n}_{1\leq i,j\leq
n})[s_{n,1}],$$
or say 
$$k[s_{n,1},\ldots,s_{n,n}]=k[G_{m_1},\ldots,G_{m_r}][s_{n,1}].$$
This gives an expression of this algebra by two different collection of homogeneous generators. Therefore each 
$G_{m_1}$ is a weighted homogeneous polynomial in 
$s_{n,1},\ldots,s_{n,n}$, and conversely, each $s_{n,i}$ is a weighted homogeneous in 
$G_{m_1},\ldots,G_{m_r}, s_{n,1}$.
 
Fix $f(Z)=Z^n+a_1Z^{n-1}+\ldots+a_{n-1}Z+a_n\in S[Z]$, or equivalently, fix a morphism $\Theta$ as in (\ref{especial}). 
 Assume that $ \nu_x (G_{m_j}(a_1,\ldots,a_n)) \geq m_j$, for $j=1, \dots, r$, at $x\in \mbox{Spec}(S)$. We claim that $x\in \delta(F_n)$. As the characteristic is zero, there is an element 
$\lambda\in S$ so that setting $Z_1=Z-\lambda$ 
 $$f(Z)=Z^n+a_1Z^{n-1}+\ldots+a_{n-1}Z+a_n=Z_1^n+b_1Z_1^{n-1}+\ldots+b_{n-1}Z_1+b_n,$$
 with $ \nu_x (b_1)\geq 1$. We claim now that
 \begin{equation}\label{eret}  \nu_x (b_i)\geq i, \ i=1, \dots , n,
 \end{equation}
 and hence that $x\in \delta(F_n)$. To this end recall firstly (\ref{lapro}), which ensures, in particular, that $\nu_x (G_{m_j}(b_1,\ldots,b_n)) \geq m_j$ for $j=1, \dots , r$. Finally, the inequalities in (\ref{eret}) follow from \ref{ark}.
 
 Conversely, we claim that if $x\in \delta(F_n)$, then $ \nu_x (G_{m_j}(a_1,\ldots,a_n)) \geq m_j$ for $j=1, \dots, r$. Let 
 $y\in F_n$ be a point that maps to $x$ in $\Spec(S)$. So $y$ in an $n$-fold point of $\Spec(B)$, and Theorem \ref{MultForm} will show that $y$ is the unique point of the fiber, and that the local rings, say $B_y$ and $S_x$, have the same residue fields. 
 Recall that $B=S[Z]/\langle f(Z)\rangle $. In particular the class of $Z$ in the residue field of $B_y$ is also the class of some element $\lambda\in S$, at the residue field of $S_x$. Set 
 $Z_1=Z-\lambda$, and 
 $$f(Z)=Z^n+a_1Z^{n-1}+\ldots+a_{n-1}Z+a_n=Z_1^n+b_1Z_1^{n-1}+\ldots+b_{n-1}Z_1+b_n=g(Z_1).$$
 Let $M_x$ denote the maximal ideal in $S_x$, and let $k'$ denotes the residue field $S_x/M_x$. The uniqueness and rationality of the point $y$ in the fiber shows that class of $g(Z_1)$ in 
 $k'[Z_1]$ is $Z_1^n$. Therefore $g(Z_1)\in \langle  M_x, Z_1\rangle^n$ (as $y$ is an $n$-fold point of $B$), and this occurs if and only if $b_i\in M_x^i$. So again, the argument in \ref{ark} together with (\ref{lapro}) show that 
$  \nu_x (G_{m_j}(a_1,\ldots,a_n)) 
\geq m_j \ , j=1, \dots ,r.$
\endproof
\begin{parrafo} In the previous discussion we have fixed a regular ring $S$, a polynomial $f(Z)=Z^n+a_1Z^{n-1}+\ldots+a_{n-1}Z+a_n
\in S[Z]$, and we have studied equations on the coefficients that are invariant under a change of variables of the form $Z_1=Z-\lambda$ for $\lambda\in S$. This lead us to the elements 
$G_{m_i}=G_{m_i}(s_{n,1},\ldots,s_{n,n})$
in (\ref{eqdg}). There are other cases of interest to be considered:

1) If $u\in S$ is a unit then $S[Z]=S[Z_1]$ where $Z_1=uZ$, is also a change of variables.

2) When there is an element $0\neq v\in S$ so that $ a_i=v^i a'_i, \ i=1, \dots, n.$

Case 1) In this case $Z=\frac{Z_1}{u}$, and
$$f(Z)=Z^n+a_1Z^{n-1}+\ldots+a_{n-1}Z+a_n=
(\frac{Z_1}{u})^n+a_1(\frac{Z_1}{u})^{n-1}+\ldots+a_{n-1}\frac{Z_1}{u}+a_n
\in S[Z]$$
which is not monic in $Z_1$, but the associated polynomial
$$u^nf(Z_1)=Z_1^n+ua_1Z_1^{n-1}+\ldots+u^{n-1}a_{n-1}Z+u^na_n$$
is monic. The weighted homogeneous expression of $G_{m_i}=G_{m_i}(s_{n,1},\ldots,s_{n,n}) $ ensures that 
$$G_{m_i}(ua_1,\ldots,u^na_n)=u^{m_i}G_{m_i}(a_1,\ldots,a_n).$$

In particular, the ideal in the ring $S$ spanned by 
$G_{m_i}(a_1,\ldots,a_n)$ is intrinsic to the polynomial $f(X)$ and independent of {\em any} change of variable in $S[Z]$.

Case 2) In this case set formally $Z_1=\frac{Z}{v}$. 
If $K$ denotes the quotient field of $S$ this is a change of variables in $K[Z]$, and 

$$(\frac{1}{v})^nf(Z)=Z_1^n+a'_1Z_1^{n-1}+\ldots+a'_{n-1}Z_1+a'_n.$$

The same argument used above shows that 
$G_{m_i}(a'_1,\ldots,a'_n)=(\frac{1}{v})^{m_i}G_{m_i}(a_1,\ldots,a_n).$
\begin{corollary}\label{Cor54} Let $k$ be a field of characteristic zero. Let $\delta: \Spec(B)\to \Spec(S)$ be given by a finite extension of rings $S\subset B$ where 
$S$ is a smooth $k$-algebra and $B=S[Z]/\langle f(Z) \rangle$, $f(Z)=Z^n+a_1Z^{n-1}+\ldots+a_{n-1}Z+a_n\in S[Z]$. The 
Rees algebra $$\G(B/S)=S[G_{m_1}(a_1,\ldots,a_n)W^{m_1}, \dots , G_{m_r}(a_1,\ldots,a_n)W^{m_r}]$$ is intrinsic to $B$ (independent of the choice of the variable $Z$), and 
$ \Sing \G(B/S)= \delta(F_n)$
(is the image of the $n$-fold points of $\Spec(B)$).
\end{corollary}

We will show later that a closed smooth center $Y$, included in 
$F_n$ (the $n$-fold point of $\Spec(B)$), maps to a smooth center, say $\delta(Y)\subset \Spec(S)$. Moreover, we will show that $Y$ and $\delta(Y)$ are isomorphic. 
Let $\Spec(B) \leftarrow T$ denote the blow up at $Y$, and let 
$\Spec(S)\leftarrow R$ denote the blow up at the regular center $\delta(Y)$. In \ref{Th120} it will be shown that there is a natural commutative diagram 
\begin{equation}\label{drdw} \xymatrix{
 \Spec (B) \ar[d]_{\delta}& & T\ar[d]_{\delta'} \ar[ll]\\
 \Spec (S)&   & \ar[ll]R\\
 } 
\end{equation}
where $\delta':T\to R$ is finite. In addition, there is a suitable affine cover of $R$ and $T$ so that the restriction of 
$\delta'$ is a finite map of the form $\Spec(S'[Z']/\langle  g(Z')\rangle) \to \Spec(S')$, where 
$g(Z')\in S'[Z']$ is a monic polynomial of degree $n$, and a {\em strict transform} of $f(Z)\in S[Z]$ (see (\ref{tres})).
\begin{corollary}\label{Cor55}
Fix the setting as above, where $k$ is of characteristic zero and $\delta(Y)\subset \Sing(\G(B/S))$. Let $\G(B/S)_1$ be the transform of $\G(B/S)$ to $R$. There is a cover of $R$ by affine schemes, so that if $U=\Spec(S')$ is an affine chart, then $(\delta')^{-1}(U)=\Spec(B')$, $B'=S'[Z']/\langle g(Z')\rangle $, with
$g(Z')\in S'[Z']$ monic of degree $n$, and the restriction of $\G(B/S)_1$ to $U$ is the algebra $\G(B'/S')$.
\end{corollary}
Note, in particular, that in characteristic zero the image of the $n$-fold points of $T$ is the closed set $\Sing(\G(B/S)_1)$, and the same holds after any sequence of blow ups obtained over $\Spec(B)$ as above.

\end{parrafo}

\end{section}

\begin{section}{Multiplicity. Some algebraic preliminaries.}

There are several algebraic preliminaries to be mentioned  here, where we review some very classical notions of commutative algebra used to study the behavior of the multiplicity, such as  finite extension of rings, integral closure of ideals, and blow ups. 

The theorem of Rees and the theorem of Zariski formulated in 
\ref{Rees} and \ref{MultForm} respectively, will be the basic tools in our discussion. 

We then turn to our algebraic reformulation of the geometric notion of branched cover. This is done in \ref{par0} and \ref{par1}, and in \ref{pr46}, we discuss the relevance of this notion for the study of the multiplicity on an equidimensional  scheme of finite type over a perfect field.

\color{black}


\begin{parrafo}\label{par0}

If $(R,M)$ is a local ring, and if $J$ is primary for the maximal ideal, then  
$e_R(J)$ will denote the multiplicity of the ideal.
A prime ideal $q$ in a ring $B$ is said to be an {\em $n$-fold point of $\Spec(B)$}, or an $n$-fold prime, when
$e_{B_q}(qB_q)=n.$

Let $(R,M)$ be a local ring of dimension $d$, and let $J$ be an $M$-primary ideal. There is a polynomial of degree $d$ is attached to these data, called the Hilbert polynomial, so that the length $l(R/J^n)$ is given by the evaluation on $n$,  for all $n$ sufficiently big. Moreover, the leading coefficient is
$\frac{e_R(J)}{d!}$.

Given a finitely generated $R$-module $N$, then $l(N/J^nN)$ is also given by a polynomial of degree $d'(\leq d)$ for $n$ sufficiently big, and where $d'$ is the dimension of the support of $N$. The leading coefficient can be expressed as $\frac{e_N(J)}{d'!}$, for a positive integer $e_N(J)$ called the multiplicity of $N$ relative to the $M$-primary ideal $J$. 

We fix $d=\dim R$ and define the $d$-multiplicity of a module, say 
$e^{(d)}_N(J)$, to be zero 
if $d'<d$, and to be $e_N(J)$ when $d'=d$.
An important property of the $d$-multiplicity is its additive behavior: given a short exact sequence of $R$-modules
$0 \to N_1 \to N_2 \to N_3 \to 0,$
then 
\begin{equation}\label{additive}
e^{(d)}_{N_2}(J)=e^{(d)}_{N_1}(J)+e^{(d)}_{N_3}(J).
\end{equation}
Namely, the coefficient in degree $d$ of the polynomial corresponding to $N_2$ is the sum of those of $N_1$ and $N_3$ (see \cite{Mat} Prop 12. D, p 74).

The following example illustrates  an application of this property concerning the local rings of multiplicity one: local regular rings have multiplicity one but the converse does not hold. Consider the inclusion $X\langle Y,Z\rangle\subset \langle X\rangle$ in $k[X,Y,Z]$, which induces a surjection of the quotient rings, say $B_1\to B_2$, and hence an exact sequence $J\to B_1 \to B_2\to 0$. 
Here $B_1$ and $B_2$ are two dimensional rings, both corresponding to sub-schemes in $\mathbb A^3$ , $B_2$ has multiplicity one at the origin, and $J$ is supported 
in a closed set of smaller dimension. Using  (\ref{additive}) one can check that 
$B_1$ has multiplicity one at the origin, but it is not regular at this point.

 Here $B_1$ is reduced but not equidimensional. We claim that this property does hold when $B$ is a finitely generated algebra over a field $k$, and $B$ is reduced and equidimensional. Namely, if the localization at a prime ideal, say $B_q$, has multiplicity one, then $B_q$ is regular.
In fact, these conditions ensure that $B_q$ is reduced, equidimensional, and also excellent. Therefore the completion $\hat{B_q}$ is again reduced and equidimensional ((\cite{EGAIV}, (7.8.3), (vii) and (x)). Finally, if this holds and $\hat{B_q}$ has multiplicity one, a theorem of Nagata states that $B_q$ is regular (\cite{Nagata}, Theorem 40.6, p.157).
\color{black}  
%
%

A local ring $(R,M)$ is said to be formally equidimensional (quasi-unmixed in Nagata's terminology) if 
$\dim(\hat{R}/p)=\dim(\hat{R})$
at each minimal prime ideal $p$ in the completion $\hat{R}$.

\begin{theorem}\label{Rees} (\cite{Rees}) If $I\subset J$ are primary ideals for the maximal ideal in a formally equidimensional local ring $(R,M)$, then both ideals have the same integral closure if and only if $e_R(I)=e_R(J)$.
\end{theorem}

\begin{parrafo}\label{redbp}
An ideal $I$, included in $J$, is said to be a {\em reduction} of $J$, if both  have the same integral closure in $R$. A criterion due to Lipman says that $I\subset J$ is a reduction if and only if $IJ^n=J^{n+1}$ for a suitable integer $n$
(\cite{Lipman1}, p 792, Lemma (1.1)). This, in turn, has interesting consequences when studying blow ups: Fix an ideal $J$ in a ring $R$, let 
$R_J:= R \oplus J \oplus J^2\oplus\dots$  denote the Rees algebra. The blow-up of $R$ at $J$ is a projective morphism, say
$Spec(R) \leftarrow X$
where $X=Proj(R_J)$. 

 When  $I\subset J$ is a reduction , then $I\calo_X=J\calo_X$, and moreover, if
$I=\langle f_1, \dots , f_r \rangle $, and 
$J=\langle f_1, \dots , f_r, 
f_{r+1}, \dots , f_s\rangle $, then $X$ can be covered by only $r$  charts, say $Spec(B_i)$, where each
$$B_i=B\left[ \frac{f_1}{f_i}, \dots , \frac{f_s}{f_i}\right] , i=1, \dots , r,$$
is a $B$-algebra included in the localization $B_{f_i}$.
In fact, the criterion shows that the natural inclusion of Rees rings $R_I\subset R_J$ is a finite extension, which induces a finite morphism of $\Spec(R)$-schemes, say $X \to Y$, where $Y$ denotes the blow up at $I$.
The existence of this morphism, together 
 of the condition  $IJ^n=J^{n+1}$, ensure that every point in $X$ is in one of the $r$ affine charts expressed above.
\end{parrafo}
The following theorem of Zariski, combined with Theorem \ref{Rees}, will be used along this paper.
%
%
%
%
%
%

\begin{theorem}\label{MultForm} (\cite{ZS}, Theorem 24 p. 297) Let 
$(A,M)$ be a local domain, and let $B$ be a finite extension of $A$. Let $K$ denote the quotient field of $A$, and $L=K\otimes_AB$.

 Let $Q_1, \ldots, Q_r$ denote the maximal ideals of the semi-local ring $B$, and assume that $\dim B_{Q_i}=\dim A$, $ i=1, \dots ,r$. Then
$$e_A(M)[L:K] =\sum_{1\leq i \leq r} e_{B_{Q_i}}(M B_{Q_i}) [k_i:k],$$
where $k_i$ is the residue field of $ B_{Q_i}$, $k$ is the residue field of $(A,M)$, and $[L:K]=\mbox{dim}_KL$.

\end{theorem}

Note that there is a free $A$-module of rank $[L:K]$, say 
$A^{[L:K]}$, included in $B$, and 
\begin{equation}\label{zmf} 0\to A^{[L:K]} \to B\to N\to 0,
\end{equation}
is an exact sequence of $A$-modules such that $N\otimes_AK=0$. As $N$ is supported in smaller dimension, the statement derives readily from the additive formula in (\ref{additive}).

For most applications we will introduce an additional condition: 
that all non zero element in the domain $A$ be a non-zero divisor in $B$ (i.e., that $B\to B\otimes_AK$ be injective). 

\end{parrafo}

\begin{parrafo}\label{par0} 

I) A noetherian ring $B$ is said to be pure dimensional if all localizations at maximal ideals have the same dimension. It is said to be equidimensional if $\dim B/q=\dim B$ for any minimal prime $q$. The two conditions hold if we require on $B$ that all saturated chains of prime ideals have the same length.

II) In what follows we only consider rings that are excellent and comply the previous condition on the chains of prime ideals. 
This ensures that the localization at any prime ideal, say $B_p$, is formally equidimensional (a requirements in Theorem \ref{Rees}, crucial  in our discussion). 

Assume now that $S$ is, in addition to the preceding conditions, a regular domain, and that:

1) $B$ contains and is finite over $S$.

2) Non-zero elements in $S$ are non zero divisors in $B$ (i.e., $B$ is a torsion free $S$-module).
 
 The importance of this latter condition will become clear in the section. Let $K$ denote the quotient field of $S$, and let $L$ denote the total quotient ring of $B$. The assumptions ensure that  $L=B\otimes K$, and hence, that 
\begin{equation}\label{008}
Ass(B)=\{q_1, \dots, q_r\}=Min (B)
\end{equation} (the associated primes are the minimal prime ideals in $B$). 

Note that if 1) holds, $\dim B=\dim S$. So by assumption $\dim B/q=\dim S$ for any minimal prime ideal, and the condition in 2) is equivalent to 
the equality in (\ref{008}). Namely, the conditions in 1) and 2) can be replaced by:

1') $B$ contains and is finite over $S$.

2') $Ass(B)=Min (B)$.

\color{black}
\end{parrafo}

\begin{parrafo} \label{par1}

The aim of these notes is to study the behavior of the multiplicity at an equidimensional scheme $X$ of finite type over a perfect field $k$. This leads to the study of the multiplicity along primes of an equidimensional ring, say $B$, which is a finite type algebra over $k$. 

Note that $B$ is excellent, equidimensional and pure dimensional. The same holds for $S$, if it is smooth over $k$ and irreducible. We will consider the case in which

1') $B$ contains and is finite over $S$.

2') $Ass(B)=Min (B)$,

which is, of course, is a particular case of \ref{par0}.

\end{parrafo}
\begin{parrafo}\label{pr46} 
We briefly sketch here why the study of the multiplicity along primes of an equidimensional ring, say $B$, which is a finite type algebra over a perfect field $k$, relates to the particular case in which there is a smooth subalgebra $S$ included in $B$, and with the conditions stated above.
%
\begin{enumerate}
\item Given an equidimensional algebra of finite type over $k$, say $B$, we show that there is a canonically defined quotient, say $B'$, so that $\Spec(B)$ and $\Spec(B')$ have the same underlying topological space, that $B$ and $B'$ have the same multiplicity at any prime, and $B'$ fulfills the equality in (\ref{008}). This enables us to assume that the equidimensional algebra $B$ already complies with the condition $Ass(B)=Min (B)$.
\item Let $B\to C$ be an \'etale morphism of affine algebras of finite type over $k$, and let $q$ be a prime in $C$ mapping to $p$. If $Ass(B)=Min (B)$ we show that after restriction to Zariski open neighborhoods at $q$ and $p$ we may assume that $C$ is equidimensional, and $Ass(C)=Min (C)$. 
\item Given $B$ and $p$ as above we will construct an \'etale morphism $B\to C$, so that $C$ contains an irreducible smooth $k$-algebra $S$, and
the conditions 1') and 2') in \ref{par1} hold for $S\subset C$ (i.e., conditions 1) and 2) in \ref{par0}). The finite extension $S\subset C$ will also comply a numerical condition: that the dimension of
$C\otimes_SK$ over the field $K$ be the multiplicity of $B_p$, where $K$ denotes the quotient field of $S$.
\item Finally, the information concerning the behavior of the multiplicity 
along primes in $C$ will descend to information of the behavior along primes of $B$.

\end{enumerate}

All these properties will be addressed in \ref{rk215}. 
\end{parrafo}

\begin{parrafo}\label{rk14} 
In the rest of this section, and in most parts of Sections 5 and 6, we will assume that $B$ is given together with $S\subset B$ as in \ref{par0}. 
 A first consequence of the multiplicity formula in Theorem \ref{MultForm} is that if $\mbox{ dim}_KL=n$, then 
$e_{B_P}(PB_P)\leq n$ for any prime ideal $P$ in $B$. We say that 
$\Spec(B)\to \Spec(S)$ is {\em transversal at $P\in \Spec(B)$}, when $e_{B_P}(PB_P)= n$. 
Note that if this condition holds for $n=1$, then the condition in 2) of \ref{par0} ensures that $B=S$.

If $B$ is complete, local, and equidimensional, then it is excellent and all saturated chains of prime ideals have the same length. Assume, in addition, that the ring contains an infinite field, then there is a regular subring $S\subset B$, and $\mbox{ dim}_KL=n$
where $n$ denotes the multiplicity of $B$. Moreover, if the zero ideal in $B$ has no embedded components, then the conditions in 
\ref{par0} hold.

\color{black}
 
 The following corollary characterizes transversal points (i.e., the $n$-fold points of $B$). 

\end{parrafo}
\begin{corollary}\label{C16} Let $P$ be a prime ideal in $B$. The following conditions 1) and 2) are equivalent:

1) $e_{B_P}(PB_P)= \mbox{ dim}_KL=n$.

2) Set ${\p}=S\cap P$.

2i) $P$ is the only prime in $B$ dominating ${\p}$ (i.e., $B_P=B\otimes_SS_{\p}$). 

2ii) $S_{\p}/{\p}S_{\p}=B_P/PB_P$.

2iii) $PB_P$ is the integral closure of ${\p}B_P$ in $B_P$.
\end{corollary}
Here $e_{S_\p}(\p S_\p)=1$, and 2iii) follows from Theorem \ref{Rees}.

\vskip 0,5cm
The multiplicity along primes of an algebra in the setting of \ref{par0}  has an interesting  compatibility when taking the reduction. This result, stated in Lemma \ref{lem28}, relies on the relation of the multiplicity of an ideal with that of its the integral closure.
\begin{remark}\label{rk2x8}

A) Fix a ring $B$. The integral closure is an operator on the ideals in $B$ that preserves inclusions: if $I \subset J$ are two ideals, then there is an inclusion of the integral closures. As the integral closure of the ideal zero is the nil-radical ideal, say $\mathcal{N}$ (the set of all nilpotent elements in $B$), it is natural to study this notion in the quotient ring 
$B_{\red}=B/\mathcal{N}$. In the setting of \ref{par0}: 
$\mathcal{N}=\cap_{i=1, \dots, r} q_i$, and the morphism from $S$ to $B_{\red}$ is injective.

B) Given $P$ and ${\p}$ as in the preceding corollary, then $P$ is a prime ideal in $B\otimes_SS_{\p}=B_{\p}$. Note that one can replace 2) by the conditions: 

2a) $S_{\p}/{\p}S_{\p}=B_P/PB_P$.

2b) $P$ is the integral closure of ${\p}B_{\p}$.

In fact, this last condition implies that $B_{\p}=B_P$.
\begin{lemma}\label{lem28} Fix $S\subset B$ and $n$ as above. Then

1) $S\subset B_{red}$ is in the setting of \ref{par0}.

2) Let $m=\dim_K(B_{red}\otimes_SK)$. A prime ideal $P$ is an $n$-fold point in $B$ if and only if $PB_{\red}$ is an $m$-fold point in $B_{red}$.
\end{lemma}
\proof The result is a consequence of the previous corollary. Note that  2) holds for a prime in $B$ if and only if it holds for that same prime in $B_{\red}$ (we use here the natural identification of prime ideals in both rings).

\end{remark}
%
%
%
%
%
%
\vskip 0,3cm
The following technical lemma is an application of the preceding Theorem \ref{MultForm}, that  will be used in the next section.
\begin{lemma}\label{lem610} Let $S\subset B' \subset B$ be a tower of finite extensions of domains, where $S$ is as in \ref{par0}, with field of quotients $K$. Let $L'=B'\otimes_SK$, and $L=B\otimes_SK$. 
 
Let $n=\dim_K(B\otimes_sK)$,  $n'=\dim_K(B'\otimes_sK)$, and fix an $n$-fold point $P$ of $\Spec(B)$. Then $P'=P\cap B'$ is a  point of multiplicity $n'$ in $\Spec(B')$. 
\proof

Let ${\p}=P\cap S$. According to \ref{C16}, $P$ is the only prime in $B$ dominating ${\p}$, and both local rings $S_{\p}$ and $B_P$ have the same residue field. So $P$ is the only prime that dominates $P'$ in $B'$, and again, both local rings have the same residue field.

Note that $L$ and $L'$ are the quotient fields of the domains $B$ and $B'$ respectively, and $[L:L']=\frac{n}{n'}$. 

On the one hand one concludes that 
$n=e_{B_P}({\p}B_P)$ and that ${\p}B_P$ is a reduction of 
$PB_P$, on the other hand:
$$e_{B'_{P'}}(P'B'_{P'}) (\frac{n}{n'})= e_{B_P}(P'B_P).$$

As ${\p}B_P\subset P'B_P$,  $P'B_P$ is also a reduction of $PB_P$, so the right hand term in this equality is $n$. Therefore $e_{B'_{P'}}(P'B'_{P'})=n'$.
\endproof

\end{lemma}

\end{section}

\color{black}

\begin{section}{Multiplicity and projection on smooth schemes.}

In this section we fix a finite extension of excellent rings $S\subset  B$ as in \ref{par0}, and we study the finite morphism $\delta: \Spec(B) \to \Spec(S)$. We assume that the general fiber has $n$ points, and we show that a strong link is established between $F_n$, the set of $n$-fold points of $ \Spec(B)$, and its image in  $\Spec(S)$. 

In Section 3 we have considered the case in which $V=\Spec(S)$ and $B=S[Z]/f(Z)$ was defined by a monic polynomial $f(Z) \in S[Z]$. There we studied properties which can be expressed in terms of the coefficients of the polynomial. 
We shall pursue in this line and Proposition \ref{16} is the main result. It  is a first step for the construction of a local presentation, and will ultimately lead us to the properties discussed in part {\bf 3)} of the introduction. 

In Corollary \ref{C17} we show that there is a natural identification $F_n \equiv \delta(F_n)$. In addition we show that $F_n$ is closed if both rings in $S\subset  B$ are finite type $k$-algebras as in \ref{par1}.
This result will later lead to a direct proof of Theorem \ref{lem51c}, as stated there, without using Bennett's results.
\begin{parrafo}\label{par2}

Fix an inclusion $S\subset B$ with the conditions given in \ref{par0}. 
An {\em algebraic presentation} of $B$, relative to the inclusion, will consist of a finite set of elements $\{\theta_1, \dots , \theta_N \}$ so that 
$B=S[\theta_1, \dots , \theta_N ]$.

We attach to  $\{\theta_1, \dots , \theta_N \}$, polynomials
$$f_1(Z), \ldots , f_N(Z) \in K[Z],$$
where $K$ is the quotient field of $S$ and each $f_i(Z)$ is the monic polynomial of smallest degree vanishing at $\theta_i$. Hence $K[Z]/\langle f_i(Z) \rangle= K[\theta_i]\subset L.$

Let $S[\theta_i]$ denote the $S$-subring of $B$ spanned by 
$\theta_i$. The inclusion $0\to S[\theta]\to B$ induces
the inclusion
$ 0\to S[\theta_i]\otimes_S K \to B\otimes_S K.$
In particular, $S[\theta_i]\otimes_S K=K[Z]/\langle f_i(Z) \rangle$, for $i=1, \dots, N$.
\end{parrafo}
\begin{lemma}\label{lem1.2}

With the assumptions  and notation as above:

1) $f_1(Z), \ldots , f_N(Z) \in S[Z]$.

2) $S[\theta_i]=S[Z]/\langle f_i(Z) \rangle$, for $i=1, \dots , N$.
\proof

1) Fix an index, say $i=1$. As $K[Z]/\langle f_1(Z) \rangle= K[\theta_1]\subset L$ is a finite extension, there is a surjection of the $r$ prime ideals $\{Q_1, \dots, Q_r\}$ in $L$ (induced by the $r$ primes in (\ref{008}), to the prime ideals in $K[\theta_1]$. These, in turn, are in one to one correspondence with the irreducible factors 
of $f_1(Z)$ in $K[Z]$. So there is an integer, say $s$, $s\leq r$, and irreducible monic polynomials, say 
$g^{(1)}(Z), \dots , g^{(s)}(Z),$
so that 
$$f_1(Z)=(g^{(1)}(Z))^{h_1} \cdots (g^{(s)}(Z))^{h_s}$$
in $K[Z]$, for some positive integers $h_1, \dots ,h_s$.

For the proof of 1) it suffices to check that each irreducible factor 
$g^{(i)}(Z)$ has coefficients in $S$.

There is an inclusion $\langle f_1(Z)\rangle \subset \langle g^{(1)}(Z)\rangle$ in $K[Z]$, and the latter is a prime ideal. 
Assume that the induced prime in 
$K[\theta_1]$ is dominated by the prime 
$Q_1$ in $L$.

Recall that $Q_1$ is in natural correspondence with the minimal prime ideal $q_1\subset B$, and that $q_1$ intersects $S$ at zero (see \ref{par1}). Hence, a finite extension of domains is given by $S\subset B/q_1$, and this induces, by localization, the inclusion of fields $K\subset B_{q_{1}}/q_1B_{q_{1}}=B/q_1\otimes_SK$. In addition, if $\overline{\theta}_1$ denotes the class of $\theta_1$ in 
$B/q_1$ (or say in $B_{q_{1}}/q_1B_{q_{1}}$), then $g^{(1)}(Z)\in K[Z]$ is the minimal polynomial of $\overline{\theta}_1$.
Finally, as $S$ is normal, $g^{(1)}(Z)\in S[Z]$ (see e.g. \cite{Mat1}, Theorem 9.2 , p.65).

2) Consider again the index $i=1$. Let $D$ denote the degree 
of the monic polynomial $f_1(Z)$ in $S[Z]$. Let $J\subset S[Z]$
denote the kernel of the surjective homomorphisms 
$S[Z]\to S[\theta_1](\subset B)$.
The inclusion $B\subset B\otimes_SK$, given by the condition 3) in \ref{par1}, ensures that $\langle f_1(Z) \rangle \subset J$. The ring 
$S[Z]/\langle f_1(Z) \rangle$ is a free $S$-module with basis 
$\{1, \overline{Z}, \dots , \overline{Z}^{D-1} \},$ and we claim that $S[Z]/\langle f_1(Z) \rangle=S[\theta_1]$.

Let $\overline{g}=s_0.1+ s_1 \overline{Z}+\cdots 
s_{N-1}\overline{Z}^{D-1}$ denote the class of an element $g\in J$ in $S[Z]/\langle f_1(Z) \rangle$, and recall that 
$S[\theta_1]\otimes_S K=K[Z]/\langle f_1(Z) \rangle$. If some coefficient $s_i$ is non-zero in $S$, then $\overline{g}$ would impose a non-trivial relation on $\{1, \overline{Z}, \dots , \overline{Z}^{D-1} \}$. This contradicts the fact that $\dim_K(S[\theta_1]\otimes K)=D$.
\endproof
\end{lemma} 
\begin{remark}\label{rk56} Any element, say $\theta\in B$ is integral over $S$, and hence $\theta\otimes 1\in B\otimes_S K$ has a minimal polynomial over $K$, say $f_{\theta}(Z)\in K[Z]$. The previous lemma shows that $f_{\theta}(Z)\in S[Z]$. Moreover, when there is an inclusion, say $B\subset B\otimes_SK$ (e.g. in the conditions of \ref{par1}), then $S[\theta]=S[Z]/\langle f_{\theta}(Z) \rangle$.

\end{remark}

We shall draw special attention to algebras of the form 
 $S[Z]/\langle f(Z) \rangle$, where $S$ is a regular ring as in \ref{par0}, and $f(Z)\in S[Z]$ is a monic polynomial.
Given a prime ideal ${\p}$ in $S$, $\nu_{\p}$ will denote the order 
function on $S$ defined by the local regular ring $S_{\p}$.
\begin{proposition}\label{P110} Fix $S$, as above, and an algebra of the form 
$B=S[Z]/\langle f(Z) \rangle$, where
$$f(Z)=Z^{s}+c_1^{}Z^{s-1}+ \cdots +c_{s}^{}\in 
S[Z].$$

Then: 1) The conditions in \ref{par0}
hold.

2) A prime ideal ${\p}$ in $S$ is the image of an $s$-fold point if and only if there is an element $\lambda\in S_{\p}$, so that setting 
$Z_1=Z-\lambda$ and 
\begin{equation}\label{eq109}
f(Z)=g(Z_1)=Z_1^{s}+c'_1Z_1^{s-1}+ \cdots +c'_{s}\in 
S_{\p}[Z_1](=S_{\p}[Z])
\end{equation}
then 
\begin{equation}\label{eq110}
\nu_{\p}(c'_j)\geq j, \ j=1,\dots , s.
\end{equation}
\end{proposition}
\proof Here $B$ is a finite free $S$-module, so 1) is clear. 

2) Note first that if (\ref{eq110}) holds, the hypersurface defined by $f(Z)$ in the regular ring $S_{\p}[Z]$ has 
a point of multiplicity $s$ at the prime ideal spanned by 
${\p}S_{\p}[Z]$ and the element $Z_1$. Namely, 
$$f(Z)\in\langle {\p}S_{\p}[Z], Z_1\rangle^s \setminus \langle {\p}S_{\p}[Z], Z_1\rangle^{s+1} . $$
So $\langle {\p}S_{\p}[Z], Z_1\rangle$ induces a prime ideal at the quotient
$B\otimes_SS_{\p}$, which is an $s$-fold point.

Conversely, if $B=S[Z]/\langle f(Z) \rangle$ has an $s$-fold point $P$, mapping to ${\p}$ in $S$, then $P$ is the unique prime ideal dominating ${\p}$, and the local rings 
$B_P$ and  $S_{\p}$ have the same residue fields (\ref{C16}). In particular, the class of $f(Z)$ in $S_{\p}/{\p}S_{\p}[Z]$ is of the form 
$(Z-\overline{\lambda})^s$, for a suitable 
$\overline{\lambda}\in S_{\p}/{\p}S_{\p}$. 
Fix $ \lambda \in S_{\p}$, inducing $\overline{\lambda}$ in the residue field, and let $Z_1=Z-\lambda$. Let 
$f(Z)=g(Z_1)=Z_1^{s}+c'_1Z_1^{s-1}+ \cdots +c'_{s}\in 
S_{\p}[Z_1](=S_{\p}[Z])$
and note firstly that $\nu_{\p}(c'_j)\geq 1$, $j=1, \ldots , s$. So $f(Z)\in \langle Z_1, {\p}\rangle S_{\p}[Z] $, and this maximal ideal in $S_{\p}[Z]$ induces the prime ideal  $PB_P$ in the quotient:
$$B_P=B\otimes_SS_{\p}=S_{\p}[Z]/\langle f(Z) \rangle.$$

Finally, as $S_p[Z]_{\langle Z_1, p\rangle} $ is regular, $B_P$ is an $s$-fold point if and only if 
$f(Z)\in (\langle Z_1, {\p}S_{\p}[Z]\rangle)^s $, and this holds if and only if $\nu_{\p}(c'_j)\geq j$, $j=1, \dots , s$.
%
%
\begin{parrafo}\label{par210}
Fix $S\subset B$ subject to the conditions in \ref{par0}.
Fix an algebraic presentation of $B$, say $\{\theta_1, \dots , \theta_N \}$ (i.e., $B=S[\theta_1, \dots , \theta_N ]$). 
Let $f_1(Z), \ldots , f_N(Z) \in K[Z]$ denote the minimal polynomials attached to the inclusion $S\subset B$ and the presentation. Let $d_i$ denote the degree of $f_i(Z)$. 

One can reorder $\{\theta_1, \dots , \theta_N \}$, and assume that there is an integer $M$, $1\leq M\leq N$, so that $d_i\geq 2$ for $1\leq i \leq M$, and $d_i=1$ for $M+1\leq i \leq N$.
So $\theta_i\in S$ for $M+1\leq i \leq N$, 
\begin{equation}\label{eq211}
B=S[\theta_1, \dots , \theta_M] \mbox{ and } d_i=\deg(f_i(Z))\geq2 \ , i=1, \dots ,M.
\end{equation}

\end{parrafo}
\begin{remark}\label{xrecta}{\em Multiplication by affine spaces} 
The product  $\Spec(B)\times_{\Spec(k)}{\mathbb{A}^1}$ is the affine scheme of $B[Y]$, where $Y$ is a variable over $B$. Note that 
$S[Y]\subset B[Y]$ is a finite extension, $B[Y]=S[Y][\theta_1, \dots , \theta_M]$, and the minimal polynomial of $\theta_i$ is the same 
$f_i(Z)$, for $i=1, \dots, M.$
\end{remark}
\begin{proposition}\label{16} Fix $S\subset B$ as in \ref{par0}, and fix $B=S[\theta_1, \dots , \theta_M]$ 
as above. Then, if $ \dim_K(B\otimes_s K)=n $, a point ${\p}\in Spec(S)$ is the image of an $n$-fold point of $Spec(B)$, if and only if 
${\p}$ is the image of a point of multiplicity $d_i (\geq 2)$ in 
$Spec(S[Z]/\langle f_i(Z) \rangle)$, for every index $i=1,\dots , M$.

%
%
\end{proposition}
\proof

Recall that $f_1(Z), \ldots , f_N(Z) \in S[Z]$, and that $S[\theta_i]=S[Z]/\langle f_i(Z) \rangle$, where $\theta_i$ is the class of $Z$, for $i=1, \dots , N$ ( \ref{lem1.2}). Assume first that the prime ideal ${\p}$ in $S$ is dominated by a prime ideal $q_i\subset S[\theta_i]$, of multiplicity $d_i$, for every index $i=1, \dots ,M$.  Proposition \ref{P110}, 2) says that there is an element 
$\lambda_i\in S_p$, so that setting $Z_i=Z-\lambda_i$, 
$f_i(Z)=h_i(Z_i)=Z_i^{d_i}+a_1^{(i)}Z_i^{d_i-1}+ \cdots +a_{d_i}^{(i)}\in 
S_{\p}[Z_i]=S_{\p}[Z]$, and  $\nu_{\p}(a_j^{(i)})\geq j$,  $1\leq j \leq d_i$.

Note that the class of  $Z_i=Z-\lambda_i$ at $S[\theta_i]=S[Z]/\langle f_i(Z) \rangle$ is $\theta_i-\lambda_i$, and that 
$h_i$ is the minimal polynomial of $\theta_i-\lambda_i (\in L=B\otimes_SK)$ over $K$. 


There is an element $g\in S\setminus {\p}$, so that after replacing $S$ and $B$ by $S_g$ and $B_g$ respectively, one can assume that 
$h_i(Z_i)=Z_i^{d_i}+a_1^{(i)}Z_i^{d_i-1}+ \cdots +a_{d_i}^{(i)}\in 
S[Z_i]=S[Z],$ (with coefficients in $S$)
and $\nu_{\p}(a_j^{(i))})\geq j$,  $1\leq j \leq d_i$. 
So at a  suitable affine open neighborhood of ${\p}$, $\{\theta_1, \dots , \theta_M \}$ can be replaced by 
$$\{\theta_1-\lambda_1, \dots , \theta_M-\lambda_M  \},$$
with minimal polynomials of degree $d_i (\geq 2)$, as before. 
Moreover, for $1\leq i \leq M$, the minimal polynomial are the previous $h_i(Z_i)$, with coefficients $a_j^{(i)}$, and $\nu_{\p}(a_j^{(i)})\geq j$. 

Clearly $B=S[\theta_1-\lambda_1, \dots , \theta_M-\lambda_M]$, and we observe that: 

a) Each $\theta_i-\lambda_i$ lies in the integral closure of ${\p}B$, and hence $\theta_i-\lambda_i\in \sqrt{{\p}B}$. Therefore any prime $Q$ in $B$, dominating the prime ideal ${\p}$ in $S$, contains $\theta_i-\lambda_i$, for all $i=1, \dots , M$.

b) $B\otimes_SS_{\p}(=S_\p[\theta_1-\lambda_1, \dots , \theta_M-\lambda_M])$ is a local ring, say $B_P$, and $PB_P$ is spanned by ${\p}B_P$ and the elements 
$\theta_1-\lambda_1, \dots , \theta_M-\lambda_M$. Or say,
\begin{equation}\label{cac}
PB_P=\langle  \theta_1-\lambda_1, \dots , \theta_M-\lambda_M, y_1, \dots , y_r \rangle 
\end{equation}
where $\{y_1, \dots , y_r \}$ is a regular system of parameters in 
$S_{\p}$.

c) $B_P/PB_P=S_{\p}/{\p}S_{\p}.$

Hence $P$ is the only prime dominating $S$ at ${\p}$, and $e_{B_P}({\p}B_P)=n=dim_K(L)$ (\ref{C16}).

\vskip 0,5cm

Conversely, suppose that $P$ is an $n$-fold prime in $B$, namely that $e_{B_P}(PB_P)=n=dim_KL.$

Let ${\p}$ be the prime ideal in $S$ dominated by $P$. 
Recall from \ref{C16} that:

a) $P$ is the unique prime ideal in $B$ dominating ${\p}$ (i.e., $\sqrt{{\p}B_P}=PB_P)$, so $B\otimes_SS_{\p}=B_P$.

b) $B_P/PB_P=S_{\p}/{\p}S_{\p}$.

c) ${\p}B_P$ is a reduction of $PB_P$.



The equality b) ensures that, after replacing $\Spec(S)$ by a suitable affine open neighborhood of ${\p}$, 
there are elements $\lambda_1, \dots , \lambda_M$ in $S$, so that
$$\{\theta_1-\lambda_1, \dots , \theta_M-\lambda_M , 
\theta_{M+1}, \dots , \theta_N,  \theta_{N+1}=\lambda_1, \dots , \theta_{N+M}=\lambda_M \},$$
is a new presentation, and the first $M$ elements have minimal polynomial of the form 
\begin{equation}\label{eq1}
h_i(Z_i)=Z_i^{d_i}+a_1^{(i)}Z_i^{d_i-1}+ \cdots +a_{d_i}^{(i)}\in 
S[Z_i],
\end{equation}
where $d_i\geq 2$, and
\begin{equation}\label{uyt}
\nu_{\p}(a_j^{(i)})\geq 1 \mbox{ for } j=1, \dots , d_i ,\mbox{ and }i=1, \dots , M. 
\end{equation}
To prove this last assertion,  note that as $P$ is the only prime ideal in $B$ dominating $S$ at ${\p}$, there is only one prime ideal in $S[\theta_i]$ dominating $S$ at ${\p}$, namely $P\cap S[\theta_i]$, and b) says that there is an element $\lambda_i\in S_{\p}$, so that 
$\theta_i-\lambda_i \in P$. Replacing $S$ by $S_s$, for a suitable $s\in S\setminus {\p}$, one can assume that $\theta_i-\lambda_i\in P\cap S[\theta_i]$. This proves (\ref{uyt}), taking $Z_i=Z-\lambda_i$.

We claim now that 
\begin{equation}\label{uyt1}
\nu_{\p}(a_j^{(i)})\geq j \mbox{ for } j=1, \dots , d_i ,\mbox{ and }i=1, \dots , M. 
\end{equation}
Fix an index $i$. In order to prove our claim it suffices to show that the conditions in (\ref{uyt1}) holds for the coefficients of each irreducible factor, say $g(Z_i)$, of $h_i(Z_i)$. Recall that the irreducible factors of $h_i(Z_i)\in K[Z_i]$ are also polynomials with coefficients in $S$, and the proof of Lemma \ref{lem1.2} shows that 
any such irreducible factor arises as the minimal polynomial, over $K$,  of 
the class, say $\overline{\theta_i-\lambda_i}\in B/q_{j_0}$, for a suitable minimal prime $q_{j_0}$ of $B$.

Fix $j_0$ as before, and let $\overline{B}=B/q_{j_0}$. So the irreducible polynomial 
$$g(Z_i)=Z_i^{s}+c_1^{}Z_i^{s-1}+ \cdots +c_{s}^{}\in 
S[Z_i]$$
is the minimal polynomial of $\overline{\theta_i-\lambda_i}\in \overline{B}$.
Recall that $S\subset \overline{B}$ is a finite extension. In particular there must be a prime ideal in $\overline{B}$ dominating ${\p}$.
As $\overline{B}$ is a quotient of $B$, and $P$ is the only prime in $B$ mapping to ${\p}$, such prime is also unique, and moreover, it is the class of $P$. It is worth pointing out that this observation shows, in particular, that the $n$-fold point $P$ must contain all the minimal prime ideals of $B$. 

Since $g(Z_i)$ divides $h_i(Z_i)$, the conditions in (\ref{uyt}) also hold for the coefficients $c_i$. Namely, all $c_i$ are in ${\p}$, and the claim is that 
\begin{equation} \label{eqr}\nu_{\p}(c_i)\geq i, \ \ i=1, \dots ,s.
\end{equation}

Consider now the inclusion of domains 
$$ S \subset S[\overline{\theta_i-\lambda_i}] \subset \overline{B},$$
and let $\overline{P}$ denote the class of $P$ at $\overline{B}=B/q_{j_0}$.
By construction:

i)  $\overline{P}$ dominates $S$ at ${\p}$, and dominates 
$S[\overline{\theta_i-\lambda_i}]=S[Z_i]/\langle g(Z_i) \rangle $ 
at the prime ideal, say $q \subset S[\overline{\theta_i-\lambda_i}]$.
Moreover,  locally at ${\p}$,  $q$ is spanned by ${\p}S[\overline{\theta_i-\lambda_i}]$ and the element $\overline{\theta_i-\lambda_i}$.

ii) $\overline{P}$ is the only prime ideal in $\overline{B}$ dominating 
$S[\overline{\theta_i-\lambda_i}]$ at $q$.

Here $g(Z_i)$ is in $\langle Z_i, \p \rangle (\subset S_\p[Z_i])$, and this regular prime ideal induces $qS_\p[\overline{\theta_i-\lambda_i}]$. Therefore (\ref{eqr}) holds if and only if $S[\overline{\theta_i-\lambda_i}](=S[Z_i]/\langle g(Z_i) \rangle )$ has multiplicity $s$ at  $q$, where $s$ is the degree of $g(Z_i)$.
Let $$K\subset K_1 \subset K_2$$
denote the inclusion of fields, where $K_1$ is the quotient field of 
$S[\overline{\theta_i-\lambda_i}]$, and $K_2$ is the quotient field of 
$\overline{B}$, and let $T=\dim_KK_2$. By construction

a') $\dim_K(K_1)=s=deg(g(Z_i))$.

b') ${\p}\overline{B}_{\overline{P}}$ is a reduction of $\overline{P}\ \overline{B}_{\overline{P}}$ and both have multiplicity $T$.

Finally \ref{lem610} says that the prime $q$ in $S[\overline{\theta_i-\lambda_i}]$ has multiplicity $s$.

%
%
%
%
%
\endproof

\begin{corollary}\label{pres}  Fix, as above, $S\subset B$ and $B=S[\theta_1, \dots , \theta_M ]$ , where $d_i\geq 2$, $i=1, \dots, M$ is the degree of the minimal polynomial of $\theta_i$ over $S$ (over $K$), say
$$f_i(Z)=Z^{d_i}+a_1^{(i)}Z^{d_i-1}+ \cdots +a_{d_i}^{(i)}\in S[Z],$$

 Let $P$ be a prime ideal in $B$ dominating $S$ at ${\p}$. The following are equivalent:

i) $e_{B_P}(PB_P)=\dim_K(K\otimes_SB)$.

ii) After replacing $S\subset B$ by $S_f\subset B_f$, for a suitable element $f\in S\setminus p$,  there are elements $\lambda_1, \ldots , \lambda_M$ in $S$ so that 
$PB_P$ is spanned by the ideal ${\p}B_P$ and the elements 
$\{\theta_1-\lambda_1, \dots , \theta_M-\lambda_M\}$. Moreover, 
the minimal polynomial of $\theta_i-\lambda_i$, say 
$$h_i(Z_i)=Z_i^{d_i}+b_1^{(i)}Z_i^{d_i-1}+ \cdots +b_{d_i}^{(i)}\in 
S[Z_i]=S[Z],$$
is that obtained by a change of variable on $f_i(Z)$,
and $\nu_{\p}(b_j^{(i)})\geq j$ for all index $1\leq j \leq d_i$.
Hence, the elements in $\{\theta_1-\lambda_1, \dots , \theta_M-\lambda_M\}$ are in the integral closure of ${\p}B_P$.

iii) The conditions 2i), 2ii), and 2iii), in \ref{C16}, hold.

\end{corollary}

\begin{corollary}\label{C17} Fix the notation and assumptions as in \ref{16} where $S\subset B$
is a finite extension of excellent rings as in \ref{par0}, and $\dim_KB\otimes_SK=n$.
Let $\delta: Spec(B)\to Spec(S)$ be the finite morphism, and let $F_n$ be the set of $n$-fold points of
$\Spec(B)$. Then:

1) If $P\in F_n$ and $\delta(P)={\p}$,  the dimensions and residue fields of the local rings $B_P$ and $S_{\p}$ are the same. In addition ${\p}B_P$ is a reduction of $PB_P$.

2) $\delta$ is a set theoretical bijection between points of $F_n$ and points of the image, say
\begin{equation}
F_n\equiv \delta(F_n)
\end{equation}
 So given $P\in F_n$, $P$ is the only prime ideal in $B$ dominating $S$ at $\delta(P)={\p}$.

3) If $S\subset B $ are given as in \ref{par1} (both are finite type algebras over a perfect field $k$), then $F_n$ is a closed in $Spec(B)$, and $F_n$ is homeomorphic to $\delta(F_n)$.

4) $\delta(F_n)=\Spec(S)$ if and only if $S=B_{red}$.
\end{corollary}
\proof 1) and 2) follow easily from Theorem \ref{MultForm}. 

3) Fix an index $i$, $1\leq i \leq M$. If $S$ is smooth over a field $k$, then the same holds for $S[Z]$, and it readily follows that the set of points of multiplicity 
$d_i=deg(f_i(Z))$ is closed. In fact it is the closed set in $\Spec(S[Z])$ described by an extension of the ideal spanned by $f_i(Z)$, using higher order differential. 
In particular the $d_i$-fold points form a closed set in $Spec(S[Z]/\langle f_i(Z)\rangle$ 
 and the image is closed in $Spec(S)$. Proposition \ref{16} says that the intersection of these closed sets is $\delta(F_n)$. So $F_n (=\delta^{-1}(\delta(F_n)))$ is closed in $Spec(B)$. As finite morphisms are proper one conclude $F_n$ and $\delta(F_n)$ are homeomorphic.
 
 4) If the generic point in $\Spec(S)$ is in $\delta(F)$, then 
 $K\otimes_SB$ must be local by 2). Moreover, $\Spec (B_{red})\to \Spec(S)$ is finite and birational. Since $S$ is normal, $B_{red}=S$.
 
 Conversely, if $red(B)=S$ then $B$ has a unique minimal prime, say $q$, and using Theorem \ref{MultForm} (see \ref{zmf})
one checks that the multiplicity of $B_P$ is the length $l(B_q)$, at any prime $P$ of $B$.
\endproof
\begin{corollary}\label{Cxc} Fix $S\subset B$ as in \ref{par1}. Fix an algebraic presentation $B=S[\theta_1, \dots , \theta_M ]$ as before. Assume that $k$ is of characteristic zero, and let 
$$f_i(Z)=Z^{d_i}+a_1^{(i)}Z^{d_i-1}+ \cdots +a_{d_i}^{(i)}\in S[Z],$$
be the minimal polynomial of $\theta_i$, $1\leq i \leq M$.
 
Let $\delta: Spec(B)\to Spec(S)$ be the finite morphism induced by the inclusion $S\subset B$, and let $F_n$ be the set of $n$-fold points of
$\Spec(B)$. Then $\delta(F_n)=\Sing(\mathcal{F})$, where $\mathcal{F}\subset S[W]$ is the smallest $S$-algebra, included in $S[W]$, containing the elimination algebras of each polynomial
$f_i(Z)$, $1\leq i \leq M.$
\end{corollary}

\end{section}
\begin{section}{Multiplicity, projections, and blow ups.}

Fix an extension of excellent rings $S\subset B$ as in \ref{par0},
where the generic fiber has $n$ points (i.e., $\dim_K(B\otimes_SK)=n$). The multiplicity at the local ring $B_p$ is at most $n$ (see \label{rk14} ) and let $F_n$ denote the set of primes with multiplicity $n$. 
A theorem of Nagata ensures that if $p\subset q$ are two primes and $p\in F_n$ then $q\in F_n$ (\cite{Nagata}, Th. 40.1)\color{black}. We say that an irreducible subscheme $Y\subset \Spec(B)$ is included in $F_n$ when the generic point of $Y$ is in $F_n$.

In this section we discuss the notion of {\em transversality}, introduced in \ref{rk14} and studied in Corollary \ref{C17}. The Proposition \ref{Pr10} will settle the property {\bf1b)}, mentioned in the introduction.

The second main result here is Theorem \ref{Th120}, where we address the property of stability of transversality.
More precisely, we prove that a regular center $Y$ included in the set of $n$-fold points of $Spec(B)$ produces a commutative diagram:

\begin{equation} \xymatrix{
 Spec(B)\ar[d]^-{\delta} & & X\ar[d] ^-{\delta_1}\ar[ll]^-{\pi}\\
 Spec(S)&   & \ar[ll]^-{\pi'}R\\
 } 
\end{equation}
where $\pi$ and $\pi'$ are the blow-ups at $Y$ and $\delta(Y)$ respectively, and $\delta_1$ is a finite morphism.

We will also show that $R$ can be covered by  
affine charts, say 
$Spec(S_1), Spec(S_2), \dots , Spec(S_r),$
so that the restriction of $\delta_1$ to the inverse image of each $Spec(S_i)$ is also affine.  Therefore $X$ can be covered by affine charts $Spec(B_1), Spec(B_2), \dots , Spec(B_r),$
where each $Spec(B_i)$ maps to $Spec(S_i)$. In addition, each restricted morphism, say
$Spec(B_i)\to Spec(S_i),$
is in the setting of $Spec(B)\to Spec(S)$: Note here that $\pi'$ is birational so the quotient field of each $S_i$ is $K$. The claim is that $S_i\subset B_i$ is finite, the two conditions in \ref{par0} hold, and  
$\mbox{dim}_K(B_i\otimes_{S_i}K)=n$ (same $n$).

The Theorem \ref{Th120} will also state properties on the polynomial equations introduced in the previous section. These properties will ultimately show that a local presentation can be given by these equations. Some preliminary results for the proof of this main theorem are gathered in \ref{C116}. 
 The next theorem will be used in the proof of Proposition  \ref{Pr10}.
 We finally extract some conclusions, at the end of this section, concerning the behavior of the multiplicity along points of an equidimensional scheme of finite type over a perfect field 
 (see 
{\em Main global results}  starting in \ref {rk215}).

%
 \color{black}


\begin{theorem}\label{ThH} (see, e.g. \cite{Herrmann}, Th 10.14) Let $(R,M)$ be a noetherian local ring, and let 
 $\{ x_1, \dots, x_m\}$ be elements in the maximal ideal.   
  Let $\overline{X_i}$ denotes the class of $x_i$ in $M/M^2$.
 
 The following are equivalent:
 
i) $\langle x_1, \dots, x_m\rangle$ is a reduction of $M$.

ii) $gr_M(R)/\langle \overline{X_1}, \dots, \overline{X_m}\rangle$
 is a graded ring of dimension zero.
 
\end{theorem}

\begin{lemma}\label{L19}
Let $S\subset B$ be as in \ref{par0}, and set $n=\mbox{dim}_K(B\otimes_SK)$. Fix an $n$-fold prime ideal $P$ in $B$, and let $\p=P\cap S$. The natural morphism of graded rings
$$gr_{\p}(S_{\p}) \to gr_P(B_P)$$
is an inclusion, and a finite extension of rings.
\end{lemma}
\proof
Let $l=\mbox{dim } S_{\p}=\mbox{dim } B_P$, and let $\{ x_1, \dots, x_l\}$ be a regular system of parameters of $S_{\p}$. Then
$gr_{\p}(S_{\p})=k'[X_1, \ldots , X_l]$ 
is a polynomial ring, where $k'$ denotes the residue field 
of $S_{\p}$, and $X_i$ is the class of $x_i$ in ${\p}S_{\p}/{\p}^2S_{\p}$, $1\leq i \leq l$. Let $\overline{X_i}$ denote the class of $x_i$ in $PB_P/P^2B_P$. Since $\langle x_1, \dots, x_l\rangle B_P$ is a reduction of $PB_P$ (\ref{C17}, 1)),  
 \begin{equation}\label{kj} gr_P(B_P)/\langle \overline{X_1}, \dots ,\overline{X_l}\rangle
 \end{equation}
is zero dimensional. Recall that $k'$ is also the residue field of $B_P$, and hence (\ref{kj}) is a finite dimensional $k'$-vector space. This implies that $gr_P(B_P)$ is a finite graded module over $gr_{\p}(S_{\p})$, as one can lift an homogeneous basis of (\ref{kj}) to homogeneous generators of $gr_P(B_P)$.
The homomorphism of $gr_{\p}(S_{\p})$ in $gr_P(B_P)$ is injective since both graded rings have the same dimension.

One can also proof this this result using Lipman's criterion at $B_P=B_\p$, namely that the Rees ring of the ideal $PB_{\p}$ is a finite extension of the Rees ring of $\p S_{\p}$ over $S_{\p}$ (see \ref{redbp}). 
\endproof
\begin{proposition}\label{Pr10} Fix the setting and notation as in \ref{C17}. 
Let $Y$ be a reduced irreducible subscheme in $Spec(B)$, and 
assume that $Y\subset F_n$. Then $Y$ is regular if and only if $\delta(Y)$ is regular. 
\end{proposition}
\proof 
Recall that $\delta$ is finite, and that $Y$ and $\delta(Y)$ have the same total quotient field (\ref{C17}). As regular schemes are normal, if $\delta(Y)$ is regular, then $Y$ and $\delta(Y)$ are isomorphic and $Y$ is regular.

Assume now that $Y$ is regular. Let $P_1(\subset B)$ be the generic point of $Y$, and 
$\p_1=P_1\cap S$. Fix $\overline{B}=B/P_1$ and $\overline{S}=S/\p_1$. $Spec(\overline{B})$ is the closed regular subscheme  of 
$Spec(B)$ with underlying space $Y$, and $Spec(\overline{S})$ is the closed subscheme  of 
$Spec(S)$ corresponding to $\delta(Y)$.

As $\delta: Spec(B) \to Spec(S)$ is finite, so is the restriction, say $\overline{\delta}: \Spec(\overline{B}) \to \Spec(\overline{S}).$
The claim is that $\overline{\delta}$ is an isomorphism.
Corollary \ref{C17} ensures that this finite morphisms is birational, and that it induces a set-theoretical bijection of the underlying topological spaces. Therefore,  the claim would follow if we  show that for every prime ideal $\overline{P}$ in $\overline{B}$: 
\begin{equation}\label{eqb} \overline{S}_{\overline{\p}}=\overline{B}_{\overline{P}}
\end{equation}
where $\overline{\p}=\overline{P}\cap \overline{S}$.
%
%
There is of course an inclusion of local domains $\overline{S}_{\overline{\p}}\subset \overline{B}_{\overline{P}}$, which, in addition, is birational and finite. The claim is that this inclusion is surjective. To this end it suffices to show that the morphism of graded rings, say
$gr_{\overline{\p}\overline{S}_{\overline{\p}}}(\overline{S}_{\overline{\p}})\to gr_{\overline{P}\overline{B}_{\overline{P}}}(\overline{B}_{\overline{P}})$ is surjective. In fact, this would imply, for instance, that both local rings have the same completion (\cite{AM}, Lemma 10.23, p112).
Let $P(\supset P_1)$ be the prime ideal in $B$ that induces $\overline{P}$ in
$\overline{B}$. Then $P\in F_n$ (recall that $F_n$ is closed), and dominates $S$ at the prime ideal, say $\p$, which induces a prime, say $\overline{\p}$ in
$\overline{S}$. 

There are two natural homomorphism of local rings, namely
$S_{\p}\to B_P$ and $S_{\p}\to \overline{B}_{\overline{P}}$. 
As the latter factors through $\overline{S}_{\overline{\p}}\to \overline{B}_{\overline{P}}$,  it suffices to prove that $S_{\p}\to \overline{B}_{\overline{P}}$ is surjective. We shall prove this by showing that
\begin{equation}\label{eq114}
gr_{{\p}S_{\p}}(S_{\p})\to gr_{\overline{P}\ \overline{B}_{\overline{P}}}(\overline{B}_{\overline{P}})
\end{equation}
is surjective.
Let $\{ x_1, \dots, x_l\}$ be a regular system of parameters in $S_{\p}$. The ideal $\langle x_1, \dots, x_l \rangle B_P$ is a reduction of the maximal ideal $PB_P$ (Corollary \ref{C17}, 1)), in particular $\langle x_1, \dots, x_l \rangle \overline{B}_{\overline{P}}$ is a reduction of the maximal ideal in $ \overline{B}_{\overline{P}}$.
We finally apply Theorem \ref{ThH}. Both $gr_{{\p}S_{\p}}(S_{\p})$ and 
$gr_{\overline{P}\overline{B}_{\overline{P}}}(\overline{B}_{\overline{P}})$ are, by assumption, polynomial rings over the same field. In this setting, the condition ii) of that theorem can only hold when (\ref{eq114}) is surjective.
\endproof

The following technical lemma will be used in our proof of the Theorem \ref{Th120}. It parallels the result in Proposition 5.9 in \cite{BeV1}, formulated now in the non-embedded context.
\begin{lemma}\label{C116} Let $P_1$ be the prime ideal in $B$ which is the generic point of the smooth irreducible scheme $Y$ included in the 
set $F_n$ of $n$-fold points. Fix $P\in Y$, or equivalently, a prime ideal $P$ in $B$ containing $P_1$. Let $\p_1$ and $\p$ be the prime ideals in $S$ dominated by $P_1$ and $P$ respectively. 
Let $B=S[\theta_1, \dots , \theta_M ]$, where each $\theta_i$ has a minimal polynomial over $K$ of degree $d_i \geq 2$.

Then, after replacing $S\subset B$ by $S_f\subset B_f$ for a suitable element $f\in S\setminus \p$, there are elements $\lambda_1, \ldots , \lambda_M$ in $S$ so that the minimal polynomial of each $\theta_i-\lambda_i$ is of the form 
$$h_i(Z_i)=Z_i^{d_i}+a_1^{(i)}Z_i^{d_i-1}+ \cdots +a_{d_i}^{(i)}\in 
S[Z_i]=S[Z]$$
(of the same degree $d_i$ as the minimal polynomial of $\theta_i$), and 
\begin{equation}\label{654}\nu_{\p_1}(a_j^{(i)})\geq j,
\end{equation} for $1\leq j \leq d_i$, where 
$Z_i=Z-\lambda_i$. 
In particular:

1)  $\nu_{\p}(a_j^{(i)})\geq j$, for all index $1\leq j \leq d_i$.

2) $\{\theta_1-\lambda_1, \dots , \theta_M-\lambda_M\}$ are in 
$PB_{P}$, $B_P=B\otimes_SS_{\p}$, and as $P$ is an $n$-fold point  $PB_{P}$ is the integral closure of $\p B_P$.

3) $\{\theta_1-\lambda_1, \dots , \theta_M-\lambda_M\}$ are in  the integral closure of 
$\p_1B_P=\p_1(B\otimes_SS_\p)$.
\end{lemma}
\begin{corollary}\label{Cor116}
Fix a regular system of parameters 
$\{x_1, \ldots , x_r, \ldots ,x_s\}$ at $S_{\p}$, so that $\p_1S_{\p}=\langle x_1, \ldots , x_r  \rangle$. Then

i) $PB_P=\langle \theta_1-\lambda_1, \dots , \theta_M-\lambda_M, x_1, \ldots , x_r, \ldots ,x_s\rangle$, and 

ii) $P_1B_{\p}(=P_1B_P)=\langle \theta_1-\lambda_1, \dots , \theta_M-\lambda_M, x_1, \ldots , x_r\rangle$.
\end{corollary}
%

\proof i) is clear from the construction. As for ii), one concludes from the lemma that  $$\langle \theta_1-\lambda_1, \dots , \theta_M-\lambda_M, x_1, \ldots , x_r \rangle\subset P_1B_\p (=P_1B_P).$$
This, and the fact that $B_P/P_1B_P=S_\p/\p_1S_\p$ is a local ring with maximal ideal generated by the classes of the element $\{x_{r+1}, \dots, x_s\}$, imply that the inclusion is an equality (see (\ref{eqb})).

%
%
\proof  (Of the lemma.) Note that $\p_1\subset \p$ in $S$, so if we prove (\ref{654}), then 1) holds. Here $\p_1$ can be identified with a regular prime in $S_\p$, so $\nu_{\p_1}(a_j^{(i)})\geq j$ if and only if $a_j^{(i)}\in \p_1^i S_\p$ (powers and symbolic powers coincide for a regular prime in a regular ring). Therefore, both 2) and 3) would also follow from the condition $h_i(\theta_i-\lambda_i)=0$ (see also Corollary \ref{C17}, 1)).

We prove now the existence of equations with the conditions in (\ref{654}). As a first step we apply Corollary \ref{pres} for the prime $P$ in $B$. This only tells us that the elements 
$\lambda_1, \ldots , \lambda_M$ con be chosen so that each
$\theta_i-\lambda_i$ has a minimal polynomial, say 
$$h_i(Z_i)=Z_i^{d_i}+a_1^{(i)}Z_i^{d_i-1}+ \cdots +a_{d_i}^{(i)}\in 
S[Z_i]=S[Z] \ (Z_i=Z-\lambda_i),$$
and $\nu_{\p}(a_j^{(i))})\geq j$, for all index $1\leq j \leq d_i$.  
These latter inequalities, and the degree $d_i$, are not affected if  $\theta_i-\lambda_i$  is replaced by 
$\theta_i-\lambda_i-\lambda'_i,$ as long as we choose  
$\lambda'_i \in \p S_{\p}$. 

We claim that, after a modification of $\lambda_i$ of this form, it may be assumed that  $\theta_i-\lambda_i\in P_1$: Recall that $\theta_i-\lambda_i\in PB_P$ (\ref{pres}), and that $B_P/P_1B_P=S_{\p}/\p_1S_{\p}$ 
(\ref{Pr10}). By choosing  an element $\lambda'_i\in  \p S_{\p}$ which induces the class of 
$\theta_i-\lambda_i$ in $S_{\p}/{\p}_1S_{\p}$, we may assume that 
 $\theta_i-\lambda''_i\in P_1B_P(\subset PB_P)$, for $\lambda''_i= \lambda_i+\lambda'_i$. 
 
 So assume now that $\theta_i-\lambda_i\in P_1(\subset P)$. In particular, $P_1$ is a prime ideal in $B_P$ which dominates the local ring $S_\p[\theta_i-\lambda_i]=S_\p[Z_i]/\langle h_i(Z_i)\rangle$ at a prime ideal which contains $\p_1S_{\p}[\theta_i-\lambda_i]$ and the element $\theta_i-\lambda_i$. Therefore, the morphism $S_{\p}[Z]=S_{\p}[Z-\lambda_i] \to B$ maps $Z-\lambda_i$ into the prime ideal $P_1$. Moreover,  as $S_{\p}\subset S_{\p}[\theta_i-\lambda_i]\subset B_P=S_{\p}\otimes_SB$ are finite extensions, this is the only prime of 
$S_{\p}[\theta_i-\lambda_i]$ dominating $S$ at ${\p}_1$.  
So after localization, say $S_{\p} \to S_{\p_1}$, it follows that 
 the class of $h_i(Z_i)$ in 
 $(S_{\p_1}/\p_1S_{\p_1})[Z_i]$ is included in the ideal 
 spanned by $Z_i$ in $S_{\p_1}/\p_1S_{\p_1}[Z_i]$ , moreover, as $P_1\cap S_{\p}[\theta_i-\lambda_i]$ is the unique prime dominating $\p_1$, it follows that the class of $h_i(Z_i)$ in 
 $(S_{\p_1}/\p_1S_{\p_1})[Z_i]$ is $Z_i^{d_i}$, or equivalently that $\nu_{\p_1}(a_j^{(i)})\geq 1$, for all index $1\leq j \leq d_i$. The proof of \ref{16} shows that under these conditions (namely that $P_1$ is an $n$-fold point), $\nu_{\p_1}(a_j^{(i)})\geq j$, for all index $1\leq j \leq d_i$ (see (\ref{uyt}) and  (\ref{uyt1})). 
 \endproof
 
\begin{parrafo} We introduce here notation that will be used in the formulation of the forthcoming theorem. Fix a regular ring $S$, with quotient field $K$ , and a prime ideal 
$\p_1$ of $S$ such that $S/\p_1S$ is also regular. Assume that there are elements $\{x_1, \dots, x_r\}$ in $S$, which are a regular system of parameters of $S_{\p_1}$. 
Let $Y'=V(\p_1)$ (the closure of $\p_1$ in $Spec(S)$), and let $Spec(S) \leftarrow R$ denote the blow-up with center $Y'$.
The scheme $R$ can be covered by affine charts, $U_t=Spec(S_t)$, 
\begin{equation}\label{118} S_t= S[\frac{x_1}{x_t}, \ldots , \frac{x_r}{x_t}] (\subset K), \ t=1, \dots , r.
\end{equation}

Suppose given a polynomial 
\begin{equation}\label{eq119}
f(Z)=Z^{s}+c_1Z_1^{s-1}+ \cdots +c_{s}\in S[Z]
\end{equation}
such that
\begin{equation}\label{eq120}
\nu_{\p_1}(c_j)\geq j, \ j=1,\dots , s.
\end{equation}
There is an inclusion $S\subset S_t$, and $c_j \in x^j_tS_t$, for every $ j=1,\dots , s$.

Fix a variable $Z'$. We will say that the polynomial 
\begin{equation}\label{tres}
f_1({Z'})={Z'}^{s}+\frac{c_1}{x_t} {Z'}^{s-1}+ \cdots +\frac{c_s}{x^s_t} \in S_t[{Z'}]
\end{equation}
is a {\em strict transform} of (\ref{eq119}).
\end{parrafo}
\begin{parrafo}
 Assume again that $B$ and $S$ are as in \ref{par0}, say $B=S[\theta_1, \ldots ,\theta_M]$, and let
$$f_i(Z)=Z^{d_i}+b_1^{(i)}Z^{d_i-1}+ \cdots +b_{d_i}^{(i)}\in 
S[Z]$$
denote the minimal polynomial of $\theta_i$.
Assume that $\p_1(\subset S)$ is dominated by a prime ideal $P_1$ in $B$ of multiplicity $n=\mbox{dim}_KB\otimes_SK$, and that $B/P_1 (=S/\p_1)$ is regular. Lemma \ref{C116} says that locally at any point of $Y'$, there are elements $\{\lambda_1, \dots , \lambda_M\}\subset S$ so 
that the polynomial of $\theta_i-\lambda_i$ is 
\begin{equation}\label{xyz}
h_i(Z_i)=Z_i^{d_i}+a_1^{(i)}Z_i^{d_i-1}+ \cdots +a_{d_i}^{(i)}\in 
S[Z_i],
\end{equation}
where $Z_i=Z-\lambda_i$, $h_i(Z_i)=f_i(Z)$ in $S[Z]=S[Z_i]$, 
and $\nu_{\p_1}(a_j^{(i))})\geq j$,  $1\leq j \leq d_i$. 

Given $Spec(S) \leftarrow R$ as above, a strict transform of each $h_i(Z_i)$ is given in $S_t[{Z}'_i]$, for some variable ${Z}'_i$, at every affine chart $Spec(S_t)$ of $R$ in (\ref{118}).

\end{parrafo}

\begin{theorem}\label{Th120} Fix the notations and conditions as before, in particular let $n=\mbox{dim}_K (B\otimes_SK) $. Let $d=\dim B$. Assume that $Y\subset Spec(B)$ is a regular irreducible subscheme with generic point $P_1$,  included in the $n$-fold points, and that $\dim Y<d$. Then a commutative diagram
\begin{equation}\label{edcdt} \xymatrix{
 Spec(B)\ar[d]^-{\delta} & & X\ar[d] ^-{\delta_1}\ar[ll]^-{\pi}\\
 Spec(S)&   & \ar[ll]^-{\pi'}R\\
 } 
\end{equation}
is given, where the horizontal morphisms are the blow ups at $Y$ and $\delta(Y)$ respectively, and $\delta_1$ is a finite dominant morphism. 
Moreover

1)  Given a point $\p\in \delta(Y)$, and after taking a restriction of $\Spec(S)$ and $\delta: \Spec(B) \to \Spec(S)$ to a suitable affine neighborhood of $\p$, $R$ can be covered by affine charts $U_t=\Spec(S_t)$, $t=1, \dots , r$, as in (\ref{118}), and $X$ by charts $V_t=\delta_1^{-1}(U_t)=\Spec(B_t)$, where
\begin{equation}\label{eq277}
B_t=S_t[\frac{\theta_1-\lambda_1}{x_t}, \dots , \frac{\theta_M-\lambda_M}{x_t}].
\end{equation}

2) Non-zero elements in $S_t$ are non-zero divisors in $B_t$ and 
$K\otimes_{S_t} B_t=K\otimes_{S} B$ is the total quotient field of $B_t$. In particular $n=\mbox{dim}_K (B_t\otimes_{S_t}K) $, and the condition in \ref{par0} hold.

3) The minimal polynomial of $\frac{\theta_i-\lambda_i}{x_t}$ is 
\begin{equation}\label{xyz11}
h'_i(W_i)=W_i^{d_i}+\frac{a_1^{(i)}}{x_t}W_i^{d_i-1}+ \cdots +\frac{a_{d_i}^{(i)}}{x_t^{d_i}}\in 
S_r[W_i],
\end{equation}
namely, the strict transform of (\ref{xyz}).
\end{theorem}

%
%

\proof The theorem will be proved by analyzing the statements locally. Fix a prime $P$ in $B$, containing $P_1$ (the generic point of $Y$). Let $\p_1=P_1\cap S$, $\p=P\cap S$. 
Fix a regular system of parameters 
$\{x_1, \ldots , x_r, \ldots ,x_s \}$ at $S_{\p}$, so that $\p_1S_{\p}=\langle x_1, \ldots , x_r  \rangle$. Fix also $\{\lambda_1, \dots , \lambda_M\}$ in $S$ (in a neighborhood of $p$ in $S$) so that:

i) $PB_P=\langle \theta_1-\lambda_1, \dots , \theta_M-\lambda_M, x_1, \ldots , x_r, \ldots ,x_s\rangle$, and 

ii) $P_1B_{\p}=\langle \theta_1-\lambda_1, \dots , \theta_M-\lambda_M, x_1, \ldots , x_r \rangle$ 
(see 
 \ref{Cor116}).
%
%

Since $\langle x_1, \ldots , x_r \rangle B_P$ is a reduction of 
$P_1B_P$ (\ref{C116}), the blow-up of $Spec(B_P)$ at the prime $P_1B_P$
can be covered by charts of the form $Spec(B_P^{(t)})$, where
$$B_P^{(t)}=B_P[\frac{\theta_1-\lambda_1}{x_t}, \frac{\theta_2-\lambda_2}{x_t}, \dots , \frac{\theta_M-\lambda_M}{x_t},\frac{x_1}{x_t}, \ldots , \frac{x_r}{x_t}],$$
for $t=1, \dots ,r$ (see \ref{redbp}).
Recall that $B_P=B\otimes_SS_\p$, as $P$ is an $n$-fold prime, and note that 
$$B_P^{(t)}=S_{\p}[\frac{\theta_1-\lambda_1}{x_t}, \frac{\theta_2-\lambda_2}{x_t}, \dots , \frac{\theta_M-\lambda_M}{x_t},\frac{x_1}{x_t}, \ldots , \frac{x_r}{x_t}],$$
and that
$S_{\p}[\frac{x_1}{x_t}, \ldots , \frac{x_r}{x_t}]\subset B_P^{(t)}$
is a finite extension. 

There is an element $s\in S\setminus \p$ so that, after replacing $S$ by $S_s$ and $B$ by $B_s$, we may assume that $X$ can be 
covered by affine charts of the form $Spec(B^{(t)})$

$$B^{(t)}=S[\frac{\theta_1-\lambda_1}{x_t}, \frac{\theta_2-\lambda_2}{x_t}, \dots , \frac{\theta_M-\lambda_M}{x_t},\frac{x_1}{x_t}, \ldots , \frac{x_r}{x_t}],$$
and that
$$S^{(t)}=S[\frac{x_1}{x_t}, \ldots , \frac{x_r}{x_t}]\subset B^{(t)}$$
is a finite extension. Here $S^{(t)}$ is a regular irreducible algebra with quotient field $K$. Let us draw attention on the fact that $B^{(t)}(\subset B_{x_t})$, is also included in $B\otimes_SK$, in particular non-zero elements in $S^{(t)}$ are non-zero divisors in $B^{(t)}$, and $S^{(t)}\subset B^{(t)}$  are as in \ref{par0}. This enables us to use the result in Remark \ref{rk56} to study minimal polynomials. Note first that 1) and 2) follow from this construction.
Recall that the minimal polynomial of $\theta_i-\lambda_i$ over $S$ (over $K$), are of the form 
$$h_i(Z_i)=Z_i^{d_i}+a_1^{(i)}Z_i^{d_i-1}+ \cdots +a_{d_i}^{(i)}\in 
S[Z_i]=S[Z],$$
and $\nu_{\p_1}(a_l^{(i)})\geq l$, for all index $1\leq l \leq d_i$.
So $x_j^l$ divides $a_l^{(i)}$ in the regular ring $S^{(j)}$.
To prove 3) note simply that 
$h'_i(Z_i)=Z_i^{d_i}+\frac{a_1^{(i)}}{x_j}Z_i^{d_i-1}+ \cdots +\frac{a_{d_i}^{(i)}}{x_j^{d_i}}\in 
S^{(j)}[Z_i]$
is the minimal polynomial of $\frac{\theta_i-\lambda_i}{x_j}$ over 
$S^{(j)}$ (over $K$).
%
%
\begin{corollary}\label{Cxc1} Assume that $S$ is smooth over a field $k$ is a field of characteristic zero, and let 
$$f_i(Z)=Z^{d_i}+a_1^{(i)}Z^{d_i-1}+ \cdots +a_{d_i}^{(i)}\in S[Z],$$
be, as in \ref{redbp}, the minimal polynomial of $\theta_i$, $1\leq i \leq M$.
Let $\mathcal{F}\subset S[W]$ be, as in \ref{Cxc}, the $S$-algebra attached to $S\subset B$ and the presentation $B=S[\theta_1, \dots, \theta_M]$ (generated by the elimination algebras of each polynomial
$f_i(Z)$, $1\leq i \leq M$).
Then $\Spec(S)\leftarrow R$ induces a transformation of $\mathcal{F}$ to a Rees algebra, say $\mathcal{F}_1$ over $R$. Moreover, the restriction of 
$\mathcal{F}_1$ to the affine chart $\Spec(S_t)$ is the Rees algebra attached to $S_t\subset B_t$ and the presentation in (\ref{eq277}).
\end{corollary}
The corollary follows from \ref{Cor55}. Note that the theorem shows that this property of $\G$ also holds for a sequence of transformation over $\Spec(B)$ with centers included in the $n$-fold points.

\begin{remark}\label{r38} 
{\bf 1)}  The fact that  $\dim Y< \dim \Spec(B)$ ensures that the blow up 
induces an isomorphism on the height zero prime ideals of $B$, namely over each $B_{q_i}$, $i=1, \dots , r$, in  (\ref{008}).

{\bf 2)}  If $B$ is replaced by the reduced ring $B_{\red}$, and if 
$m =\dim_K (B_{\red}\otimes_SK)$, then $Y$ is included in the $m$-fold points of $B_{\red}$ (\ref{lem28}). Moreover, there is a natural compatibility of 
the diagram (\ref{edcdt}) with reductions. In other words,
\begin{equation}\label{edcdt2} \xymatrix{
 \Spec(B_{\red})\ar[d]^-{\beta} & &X_{\red}\ar[d] ^-{\beta_1}\ar[ll]^-{\pi}\\
 \Spec(S)&   & \ar[ll]^-{\pi'}R\\
 } 
\end{equation}
is the outcome of blowing up $B_{\red}$ at $Y$. 
In particular:

i)  $X_{\red}$ is covered by the affine charts 
$(B_t)_{\red}$ , $t=1, \dots, r$, for each $B_t$ as in (\ref{eq277}).

ii) The set of $n$-fold points in $X$ coincides with the set of $m$-fold points of $X_{red}$.

iii) The Theorem shows that the same holds after blowing up several times over $\Spec(B)$, as long as the multiplicity is $n$  along the centers, or equivalently, over $\Spec(B_{red})$ as long as the centers have multiplicity $m$. 
\end{remark}
%
%


\begin{parrafo}\label{rk215} {\em Main global results.}
%
We study now the behavior of  the function 
$ \mult_X: X \to \mathbb{N}$, for an equidimensional scheme $X$ of finite type over a perfect field $k$. $X$ can be covered by finitely affine schemes of the form $\Spec(B)$, where $B$ is an equidimensional algebra of finite type over $k$. Up to this point, in our discussion, we have drawn attention to algebras $B$ with the additional two conditions stated in \ref{par1}. We will first show that that in order to study the multiplicity 
we may assume that $Ass(X)=Min(X)$, and then we prove that $X$ can be covered, \'etale locally, by charts that fulfill these two conditions. We finally extract conclusions in \ref{lem51c} and \ref{lem512}.


We began by showing that in order to study the multiplicity along prime ideals of an equidimensional algebra $B$ as above, we may assume that $Ass(B)=Min(B)$. 
More precisely, given a pure dimensional scheme $X'$ of finite type over a perfect field $k$, we claim here that there is a uniquely defined closed subscheme, say $X\subset X'$, so that:

a) $X$ and $X'$ have the same underlying 
topological space, 

b) $ \mult_X= \mult_{X'}$, and 

c) $Ass(X)=Min(X)$. 

This would show that for the purpose of studying the stratification of singularities of $X$, defined by the level sets of function $\mult_X$, we may always assume that $Ass(X)=Min(X)$. 


Fix an equidimensional algebra $B$, and let $Min(B)=\{ q_1, \dots , q_r\}$. Let $J_i= ker (B \to B_{q_i}) \ i=1, \dots, r.$
Each $J_i$ is the $q_i$-primary component of the ideal zero in $B$. Finally set 
$J=\cap_{i=1, \dots, r} J_i$, and $B'=B/J$. It follows from the construction that $Ass(B')=Min(B')=\{ q_1, \dots , q_r\}$ and 
 that $B_{q_i}=B'_{q_i}$, $i=1, \dots, r$.
In particular $B$ and $B'$ have the same underlying topological space. Finally note that the additive property in (\ref{additive}), applied to the  localization of the short exact sequence 
$$0 \to J \to B \to B' \to 0$$
at any prime ideal $P$, shows that $e_{B_P}(PB_P)=e_{B'_P}(PB'_P)$. This proves a), b), and c), as the same holds at any any affine chart of $X$, and settles the previous claim for $X\subset X'$.

We claim now that the condition in c) is compatible with \'etale topology.  Namely, if $Min(B)=Ass(B)$ and $B\to C$ is \'etale, then $Min(C)=Ass(C)$. 

It suffices to check this latter condition locally at any prime ideal $P$ of $C$. One can choose an element $f\in C\setminus P$ so that $C_f$ is a localization of a finite and free $B$-module (see \cite[Remark 2, pg. 19]{R}). This latter observation shows, in addition, that $X$ fulfills c) if and only if this condition holds at an affine \'etale cover of $X$.

We remark here that the proof of Theorem \ref{Th120} shows that the previous conditions on $X\subset X'$ are compatible with blow ups at equimultiple centers. Namely, let $X' \leftarrow X'_1$ and $X\leftarrow X_1$  denote the corresponding blow ups, by general properties of blow-ups we know that $X_1\subset X'_1$, and the construction of the charts given in the theorem show that a), b) and c) will also hold for $X_1$.

We finally claim that if $Ass(X)=Min(X)$, then $X$ can be covered, \'etale locally, by affine charts which fulfill the additional conditions in \ref{par1}. We refer here to \cite{BrV1}, Appendix A), and to our previous discussion, to show that 
the \'etale cover can be chosen by finitely generated $k$-algebras $C$, together with a finite extension $S\subset C$, where $S$ is smooth over $k$,  and the conditions in \ref{par1} hold.

\begin{theorem}\label{lem51c} (Dade-Orbanz) Let $X$ be an equidimensional scheme of finite type over a perfect field, then 
\begin{itemize}
\item[1)] $\mult_X: X \to (\mathbb{N}, \geq) $ 
is an upper semi-continuos function.
\item[2)] If $\alpha: X_1\rightarrow X$ is the blow up at a smooth equimultiple center, then $ \text{mult}_X(\alpha(\xi))\geq  \text{mult}_{X_1}(\xi)$ at any $\xi\in X_1$, in particular
$\max \text{mult}_X\geq \max \text{mult}_{X_1}$.
\end{itemize}
\end{theorem}
\proof
The discussion in the preceding paragraph says that it suffices to prove these results locally, in \'etale topology and in the setting of \ref{par0}. The claim in 1) follows now from the Corollary \ref{C17}, 3), and the claim in 2) follows from  Theorem \ref{Th120}, 2), and  \ref{rk14}. 

\begin{remark}\label{XXX1}

This theorem, formulated here for equidimensional scheme of finite type over a perfect field, holds in more generality. We sketch a proof  of Dade for the case in which $X$ is an excellent scheme,  and all saturated chains of  irreducible subschemes $Y_0 \subset Y_1 \dots \subset Y_d$ have the same length. 

1) There are two properties to be checked in order to prove that
$\text{mult}_X$ is upper semi-continuous:

i) Given two points, $x ,y$ in the underlying topological space, and if $x\in \overline{y}$, then $$\text{mult}_X(x)\geq \text{mult}_X(y).$$

ii) The set $\{ x\in \overline{y} / \text{mult}_X(x)=\text{mult}_X(y)\}$
contains a dense open set in $\overline{y}$.

 The property in i) follows from a result of Nagata which we have also used in our discussion. He proves that if $p$ is a prime ideal in a local ring $R$, then $e(R)\geq e(R_p)$ holds under conditions which are valid at the local rings of points  $x\in \overline{y}$ in $X$) (see \cite{Nagata1}, or Theorem 40.1 in \cite{Nagata}). 
 
 We now sketch Dade's proof of the property in ii). Assume that $R$ is an equidimensional and excellent local ring. Let $V\to \Spec(R)$ denote the blow up of $\Spec(R)$ at $p$, and let $Y\to \Spec(R/p)$ be the restriction to $\Spec(R/p)\subset \Spec(R)$. The conditions on $R$ ensure that $Y$ is pure dimensional, of dimension $h(p)-1$ locally at any closed point.
 
 A theorem of Chevalley shows that the dimension of the fibers of $Y\to \Spec(R/p)$, corresponding to the different 
 points in $\Spec(R/p)$, is an upper semi-continuos function on the underlying topologic space of 
 this scheme (see Theorem in (13.1.3), \cite{EGAIV} No.28). 
 
 It is simple to check that the fiber at the generic point of $\Spec(R/p)$ has dimension $h(p)-1$. Dade
 proves that if the local ring $R/p$ is regular, then the condition $e(R)=e(R_p)$ holds if and only if all the fibers of the morphism $Y\to \Spec(R/p)$ have the same dimension $h(p)-1$ (\cite{D}).
 
 Fix a point $y\in X$. As $X$ is excellent there is a dense open set of points 
 $x\in \overline{y}$ where  $ \overline{y}$ is regular. The condition in ii) finally follows from this fact and Dade's observation.
 
The dimension of the closed fiber of the morphism $Y\to \Spec(R/p)$ is known as the analytic spread of $p$ in the local ring $R$, a notion introduced by Northcott and Rees (see \cite{NR}, p. 149).
%
%
%
%
%
%

2) This result is also due to Dade. He proved that the multiplicity does not increase when blowing up at equimultiple centers (\cite{D}). The proof was later simplified by Orbanz in 
\cite{Orbanz}, using a generalization of invariants introduced by Balwant Singh in \cite{BS}, and the generalized Hilbert function introduced in \cite{HerrmannAL}. 

We refer also to \cite{HerrmannOr}, \cite{Herrmann}, and \cite{Lipman2}, for detailed expositions on this development.

We end here indicating that the proof we gave for 2), in Theorem \ref{lem51c}, also applies 
for the case that $\calo_{X, \alpha(\xi)}$ is excellent, pure dimensional, and contains an infinite field. The proofs given by Orbanz and by Dade require less conditions, and are therefore more general.
\end{remark}
\color{black}

The local description of the stratification obtained the multiplicity in Corollary \ref{C17}, together with  Lemma \ref{lem28}, renders the following result.

\begin{proposition}\label{lem512} Let $X$ be an equidimensional scheme of finite type over a perfect field, then the partition of the underlying topological space obtained from the level sets of the function $\mult_X: X \to \mathbb{N}$, coincides with that  obtained from $\mult_{X_{red}}: {X_{red}} \to \mathbb{N}$.
\end{proposition}

\end{parrafo}

\end{section}

\begin{section}{Local presentations attached to the multiplicity.}

%
%
%
%
%
%


Fix a reduced equidimensional scheme $X'$, of finite type over a perfect field $k$, with maximum multiplicity $n$. We will consider  a sequence of blow ups, say 
\begin{equation}\label{3004}X'\leftarrow X'_1 \leftarrow \dots \leftarrow X'_r,
\end{equation}
where each transformation is the blow up at a smooth center included in the $n$-fold points. 
Local presentations provide a (local) description of the set of $n$-fold points along any sequence as before. This description, 
discussed in the Introduction (see {\bf 3)}) 
is obtained via a local inclusion of $X'$ in a smooth scheme over $k$.

Over fields of characteristic zero, local presentations enable us to {\em construct}, for each reduced scheme $X$ in the previous conditions, a resolution of singularities. In fact, in Section 8 we show that this provides a 
{\em constructive} proof of Theorem \ref{Res}, with the expected natural properties.

\begin{parrafo}
\label{blpm}{\em On the local setting and local presentations.} 
Assume that $X'$ is a reduced equidimensional scheme, of finite type over a perfect field $k$, and that $Ass(X')=Min(X')$. This latter condition is clearly fulfilled if $X'$ is reduced.
Let $n$ denote the highest multiplicity along points of $X'$, and let 
$F_n(X')$ denote the set of $n$-fold points.
\color{black}

The key point for local presentations relies on the fact that \'etale locally one can assume that that conditions in  \ref{rk14} hold (see \ref{rk215}). Namely, that locally at a point of multiplicity $n$ in $X'$ one can obtain an \'etale neighborhood, say $X\to X'$,
so that $X$ is affine, say $X=\Spec(B)$, and there is a finite morphism $\delta: X=\Spec(B)\to V=\Spec(S)$, of generic rank $n$.  
Under these conditions Corollary \ref{pres} and Theorem \ref{Th120} ensures that: 

\begin{enumerate}
\item the sequence (\ref{3004}) induces, by taking an \'etale restriction, a commutative diagram \begin{equation}
\label{gdigra}
\xymatrix@R=0pt@C=30pt{
X_0:=X\ar[dddd]^{\delta_0}     & \ar[l]
 X_1  \ar[dddd]^{\delta_1}   & \ldots  \ar[l] &   \ar[l] X_r \ar[dddd]^{\delta_r}  \\
\\
\\
\\
V_0   & 
 \ar[l] V_1      & \ldots   \ar[l] &   \ar[l] V_r
}
\end{equation}
where the vertical maps  are finite morphisms, and they induces a homeomorphism, say 
$\delta_i: F_n(X_i)\to \delta_i(F_n(X_i))$ for each index $i=1, \dots ,r.$

\item The expression $B=S[\Theta_1, \dots, \Theta_M]$ and the surjection $S[X_1, \dots, X_M]\to S[\Theta_1, \dots, \Theta_M]$ induces an embedding $X\subset W=\Spec(S[X_1, \dots, X_M])$.
and a Rees algebra
$$\mathcal{G}=\calo_W[f_1T^{n_1}, \dots,f_MT^{n_M} ]( \subset \calo_W[T]$$
where $f_i=f_i(X_i)$ is the minimal polynomial of $\Theta_i$, which is a monic polynomial $S[X_i]$ of degree $n_i$, and 
$$\Sing(\mathcal{G})=F_n(X) \subset W.$$
\item The sequence (\ref{gdigra}) and the previous inclusion $X\subset W$ induces
\begin{equation}
\label{gdigrt}
\xymatrix@R=0pt@C=30pt{
W_0:=W\ar[dddd]^{\delta'_0}     & \ar[l]
 W_1  \ar[dddd]^{\delta'_1}   & \ldots  \ar[l] &   \ar[l] W_r \ar[dddd]^{\delta'_r}  \\
\\
\\
\\
V_0   & 
 \ar[l] V_1      & \ldots   \ar[l] &   \ar[l] V_r
}
\end{equation}
and closed immersions 
\begin{equation}X_i\subset W_i  ,  \ 1, \dots, r.
\end{equation}

\item The upper row in the preceding diagram induces  transformations 
\begin{equation}\label{lilo}(W_0, \mathcal{G}_0):=(W, \mathcal{G})\leftarrow (W_1, \mathcal{G}_1)\leftarrow \dots \leftarrow (W_r, \mathcal{G}_r)
\end{equation}
and 
\begin{equation}\label{bfx}\Sing(\mathcal{G}_i)=F_n(X_i)\subset W_i ,  \ 1, \dots, r.
\end{equation}
\item Conversely, any sequence of transformation of $(W, \mathcal{G})$
, as that in (\ref {lilo}), induces a diagram (\ref{gdigrt}), and a diagram (\ref{gdigra}). In addition the equalities in (\ref{bfx}) hold.
\item The properties in 3), 4), and 5), also hold if multiplication by affine spaces are interspersed in the sequences (\ref{gdigra}) (see \ref{xrecta}).
\end{enumerate}
\begin{remark}\label{rklopr} The previous result settles the existence of local presentations mentioned in the Introduction. 
The Remark \ref{rk210}
can be used to show that the algebra $\mathcal{G}$ is well defined up weak equivalence.
\end{remark}
\end{parrafo}

\section{From local presentations to equivariant resolution of singularities.}\label{strongly_varieties} 
In this section, devoted to resolution of singularities, $X$ will be mostly a reduced equidimensional scheme, separated and of finite type over a perfect field $k$. In such case we abuse the notation and say that $X$ is a variety over $k$. We refer to \cite{BrV1}, Section 26, for further details on the next definitions and results.

Let $X$ be a variety. A {\em local sequence over $X$}  will be a sequence of morphisms: 
\begin{equation}
\label{local_X}
X=X_0\stackrel{\varphi_1}{\longleftarrow} X_1\stackrel{\varphi_2}{\longleftarrow}\ldots  \stackrel{\varphi_m}{\longleftarrow} X_m
\end{equation}
where each $X_i\longleftarrow X_{i+1}$ is either  blow up at a smooth center $Y_i\subset X_i$, or a smooth morphism of one of the following forms:
\begin{enumerate}
\item  The restriction to an open Zariski subset of $X_i$; 
\item $X_{i+1}$ is of the form $X_i\times {\mathbb A}^n_k$, and then $X_i\longleftarrow X_{i+1}$ is the projection on the first coordinates.
\end{enumerate} 

We will be interested in studying certain upper-semi continuous functions on $X$ that naturally extend 
at each step of a local sequence over $X$.
Fix a well ordered set 
$(\Lambda, \geq)$ and assume that for every variety $X$ an  upper-semi continuous function  
$$\text{F}_X: X \mapsto (\Lambda, \geq)$$ is defined. Let $\text{max F}_X$ denote the maximum value of $\text{F}_X$, and consider the closed subset of $X$,  
$$\underline{\text{Max}} \text{ F}_X:=\{\xi\in X: \text{F}_X(\xi)= \text{max F}_X\}.$$ 

A  local sequence over $X$, like that in (\ref{local_X}), is said to be {\em $F_X$-local}  if whenever $\varphi_i: X_{i+1}\longrightarrow X_{i}$ 
is the blow up at a smooth center $Y_i\subset X_i$, one has that $Y_i\subset \underline{\text{Max}} \text{ F}_{X_i}$. 

We will  say that $F_X$ is a  {\em strongly upper-semi continuous}  if :
\begin{itemize} 
\item given  any $F_X$-local sequence, 
\begin{equation}
\label{local_X_b_bis}
X=X_0\stackrel{\varphi_1}{\longleftarrow} X_1\stackrel{\varphi_2}{\longleftarrow} \ldots  \stackrel{\varphi_m}{\longleftarrow} X_m
\mbox{    \ one has that }
\end{equation}
$$\text{max F}_{X_0}\geq   \text{max F}_{X_1}\geq \ldots \geq \text{max F}_{X_m}.$$
\item If $\alpha: X'\to X$ is smooth (e.g. if it is \'etale), then $F_{X'}= F_X \cdot \alpha$.

\end{itemize}
As we will be studying  upper-semi continuous functions  defined for every variety,  we  refer to them  as {\em upper-semi continuous functions on varieties}. If $F$ is an upper-semi continuous function on varieties,  we  denote by $F_X$ the function on a  concrete $X$. 

\begin{example}  Consider the set $\Lambda=\mathbb{N}^{\mathbb{N}}$  with the lexicographic ordering. Given a closed point $\xi$ in a variety $X$, the Hilbert-Samuel function at the local ring $\mathcal{O}_{X,{\xi}}$ is a function from $\mathbb{N}$ to $\mathbb{N}$, so the graph is an element of $\Lambda=\mathbb{N}^{\mathbb{N}}$.
One can extend this function to non-closed points, say $\text{HS}_X: X\to \Lambda$, in such a way that it is upper-semi-continuos. A theorem of Bennett states that $F=HS$ is a strongly upper semi continuos function.

Another example is given by the multiplicity, namely $\text{mult}$, where the totally ordered set to be considered is $\Lambda=\mathbb{N}$.
(see \ref{lem51c}).  
\end{example}

\begin{definition}\label{d_representable} 
Let $F$ be an upper-semi continuous function on varieties.  We will say 
that {\em $F$ is globally representable} at a variety $X$, if there is an embedding in a smooth scheme 
$V^{(n)}$, 
$$X\subset V^{(n)}$$ 
and there is  an ${\mathcal O}_{V^{(n)}}$-Rees algebra ${\mathcal G}^{(n)}$ so that the following conditions hold: 
\begin{enumerate}
\item There is an equality of closed sets: 
$$\underline{\text{Max}} \ F_X=\Sing {\mathcal G}^{(n)};$$

\item Any $F_X$-local sequence  
$\begin{array}{ccccccc}
X=X_0 &  \leftarrow &  X_1 &  \leftarrow  &  \ldots & \leftarrow  & X_m
\end{array}$
with 
$$\text{max F}_X=\text{max F}_{X_0}\dots = \text{max F}_{X_{m-1}} \geq \text{max F}_{X_m}
$$
induces a sequence 
$$\begin{array}{ccccccc} 
V^{(n)}_0=V^{(n)} & \leftarrow & V^{(n)}_1 & \leftarrow & \ldots & \leftarrow V^{(n)}_m\\
{\mathcal G}_0^{(n)}= {\mathcal G}^{(n)} & & {\mathcal G}_1^{(n)} & &\ldots  & {\mathcal G}_m^{(n)}
\end{array}$$ 
and:
$$\underline{\text{Max}} \ F_{{X}_i}=\Sing {\mathcal G}^{(n)}_i, \ i=1, \dots, m-1,$$
$$\underline{\text{Max}} \ F_{{X}_m}=\Sing {\mathcal G}^{(n)}_m, \mbox{ if } \text{max } \ F_{{X}_{m-1}}=\text{max } \ F_{{X}_{m}}$$
$$\Sing {\mathcal G}^{(n)}_m=\emptyset , \mbox{ if }   \text{max } \ F_{{X}_{m-1}}>\text{max } \ F_{{X}_{m}}.$$

\item Conversely, any  ${\mathcal G}^{(n)}$-local sequence 
$$\begin{array}{ccccccc} 
V^{(n)}_0=V^{(n)} & \leftarrow & V^{(n)}_1 & \leftarrow & \ldots & \leftarrow V^{(n)}_m\\
{\mathcal G}_0^{(n)}= {\mathcal G}^{(n)} & & {\mathcal G}_1^{(n)} & &\ldots  & {\mathcal G}_m^{(n)}
\end{array}$$ 
induces an $F_X$-local sequence  
$$\begin{array}{ccccccc}
X=X_0 (\subset V^{(n)}_0) &  \leftarrow &  X_1(\subset V^{(n)}_1)  &  \leftarrow  &  \ldots & \leftarrow  & X_m  (\subset V^{(n)}_m) \mbox{ with }
\end{array}$$
$$\text{max F}_X=\text{max F}_{X_0}\dots = \text{max F}_{X_{m-1}} \geq \text{max F}_{X_m} $$
and:
$$\underline{\text{Max}} \ F_{{X}_i}=\Sing  {\mathcal G}^{(n)}_i, \ i=1, \dots, m-1,$$
$$\underline{\text{Max}} \ F_{{X}_m}=\Sing {\mathcal G}^{(n)}_m, \mbox{ if }  \text{max } \ F_{{X}_{m-1}}=\text{max } \ F_{{X}_{m}}$$
$$\Sing  {\mathcal G}^{(n)}_m=\emptyset , \mbox{ if }   \text{max } \ F_{{X}_{m-1}}>\text{max } \ F_{{X}_{m}}.$$

\end{enumerate}
\end{definition}

\begin{definition}\label{d_representable1} 
We will say that  a  strongly upper-semi-continuous   function on varieties $F$ is  {\em representable via local embedding}, if for each variety  $X$ and  every point  $\xi \in X$,   the previous definition holds  at  some \'etale neighborhood of $\xi$. 
Note that it suffices to restrict this condition to closed points of $X$.\color{black}
\end{definition} 
\begin{example} \label{ejemplo_Aroca} A Theorem of Aroca  asserts that the Hilbert-Samuel function, $\text{HS}_X$, is representable via  local embedding. 
Given a point  $\xi\in X$, so after replacing $X$ by an \'etale neighborhood  there is an embedding $X\subset V^{(n)}$ in a smooth ambient space $V^{(n)}$, and there is  an 
${\mathcal O}_{V^{(n)}}$-Rees algebra 
${\mathcal G}^{(n)}$, unique up to weak equivalence, such that 
$\Sing {\mathcal G}^{(n)}=\underline{\text{Max}} \text{ HS}_X$ is the set of point with the highest Hilbert Samuel function.
This results is discussed in detail by Hironaka in \cite{Hironaka77} where  $X\subset V^{(n)}$ and 
${\mathcal G}^{(n)}$ appear in the form of {\em idealistic exponent}. Moreover, his proof shows that $n$ can be chosen as the embedded dimension at the point.
\end{example}

\begin{example}\label{pres_mult}
It follows from Theorem \ref{lem51c} that the multiplicity along point of  variety $X$, say $\text{mult}_X:X \to \mathbb{N}$, is upper semi-continuous, and \ref{blpm} says that it
is  representable via local embedding. 

%

\end{example}

\begin{remark}  \label{lo_q}

1) From the Definition \ref{d_representable} it follows that if $F$ is globally representable at a variety $X$, so that $F_X$ is represented by a pair say $(V^{(n)}, {\mathcal G}^{(n)})$, then a resolution of $(V^{(n)}, {\mathcal G}^{(n)}, E^{(n)}=\{\emptyset\})$ induces a sequence of blowing ups on $X$, 
\begin{equation}
\label{seq_X_bajada}
\xymatrix{
X=X_0  & \ar[l] X_1 & \ldots \ar[l] & X_m 
\mbox{, and  }}
\end{equation} 
$$\text{max F}_X=\text{max F}_{X_0}\dots = \text{max F}_{X_{m-1}} >  \text{max F}_{X_m}.
$$
Here the data are $F$ and  $X$, and ideally, or say, one property in the way of globalization would be that the sequence (\ref{seq_X_bajada})  should not depend on the particular choice of the pair  $(V^{(n)}, {\mathcal G}^{(n)})$.  

2) Let $(V^{(n)}, {\mathcal G}^{(n)})$ be as in Definition \ref{d_representable}. Note that if 
$(V^{(n)}, {\mathcal G}^{(n)})$ and 
$(V^{(n)}, {\mathcal G'}^{(n)})$ are weakly equivalent, then the first can be replaced by the second. 

Over fields of characteristic zero we know how to construct 
a resolution of $(V^{(n)}, {\mathcal G}^{(n)}, E^{(n)}=\{\emptyset\})$. This construction is compatible with equivalence, so that if 
$(V^{(n)}, {\mathcal G}^{(n)}, E^{(n)}=\{\emptyset\})$ and 
$(V^{(n)}, {\mathcal G'}^{(n)}, E^{(n)}=\{\emptyset\})$ are equivalent, both undergo the same constructive resolution.

In general we will consider $F$ to be  representable via local embedding, as is the case for the examples given in \ref{ejemplo_Aroca}  and in \ref{pres_mult}, and we note that basic objects arise only \'etale locally. We aim now to indicate, in the next Theorem \ref{estratificar_X}, that the constructive resolution of basic objects of the form $(V^{(n)}, {\mathcal G}^{(n)}, E^{(n)}=\{\emptyset\})$  already allow us to produce a {\em global procedure}, at least for varieties over fields of characteristic zero. Namely, if $F$ is representable via local embeddings, then for each $X$ of characteristic zero one can {\em construct} a sequence with the property in (\ref{seq_X_bajada}), even if $X$ is not globally embedded.
\end{remark} 
\end{section}

\begin{parrafo}{\em Equivariance} \label{about_equivariance}
A variety is a scheme of finite type over a field $k$, obtained by patching affine schemes of algebras of finite type over $k$, with some additional conditions. We say that a variety induces an abstract scheme simply by neglecting the structure over the field $k$.

Let $F$ be a strongly upper-semi continuous function on varieties. We say that $F$ is {\em equivariant}  if whenever 
$\Theta: X'\longrightarrow X$
is an isomorphism of the underlying abstract schemes, then
$$F_{X'}(\xi)=F_{X'}(\Theta(\xi))$$
for all $\xi \in X'$.  Both, the Hilbert-Samuel function and the  multiplicity  are examples of equivariant strongly upper-semi continuous functions on varieties.  In fact, isomorphic local rings have the same Hilbert-Samuel function and the same multiplicity.

Assume, in addition, that $F$ is locally representable. Fix a point $\xi\in X$ and an isomorphism $\Theta: X' \to X$, and set $\xi'=(\Theta)^{-1}(\xi)$.
Note that given an \'etale neighborhood of $X$ at $\xi$, one can obtain, via $\Theta$, an isomorphic \'etale neighborhood of $X'$ at $\xi'$. 

Suppose now that $F_X$ is represented by a pair $(V^{(n)}, {\mathcal G}^{(n)})$. After replacing $X$ by a suitable \'etale neighborhood of $\xi$ (and  $X'$ by a suitable \'etale neighborhood of $\xi'$ ), we may assume that $X$ is globally represented by $(V^{(n)}, {\mathcal G}^{(n)})$, and $\Theta$ is an isomorphism of abstract schemes. We obtain 
$$X' \to X\subset V^{(n)},$$
in particular an embedding, say $X' \subset V^{(n)}$. 

As $F$ is equivariant we observe that, via this embedding,  $\underline{\text{Max}} \ F_{X'}=\Sing {\mathcal G}^{(n)}$. Moreover, for any $F_X$-local sequence  
$\begin{array}{ccccccc}
X=X_0 &  \leftarrow &  X_1 &  \leftarrow  &  \ldots & \leftarrow  & X_m
\end{array}$
and taking successively fiber products, we get a diagram 
$$\xymatrix{
X=X_0  & \ar[l] X_1 & \ldots \ar[l] &  \ar[l] X_m \\
X'=X'_0 \ar[u]_{\theta=\theta_0} & \ar[l] X'_1 \ar[u]_{\theta_1} & \ldots \ar[l] &  \ar[l] X'_m \ar[u]_{\theta_m}
}$$
and isomorphisms $\Theta_i$, $i=0, \dots m$, so that 
$$\text{max F}_{X'_0}\dots = \text{max F}_{X'_{m-1}} > \text{max F}_{X'_m}
$$
(and 
$\text{max F}_{X_0}\dots = \text{max F}_{X_{m-1}} >  \text{max F}_{X_m}
$). 

Using these arguments one checks that $(V^{(n)}, {\mathcal G}^{(n)})$ represents $F_{X'}$ via the previous embedding $X' \subset V^{(n)}$. 

The equivariance of constructive resolution can be stated as follows (see \cite{Villa92}) .

\color{black}

\end{parrafo}
%
%
%
%

\begin{theorem} \label{estratificar_X}  Let $F$ be a strongly upper-semi continuous function on varieties.  Assume, in addition that  $F$ is  representable via local embedding. Then, given a variety $X$, defined over a perfect field $k$, the following (global) results hold: 
\begin{enumerate}
\item  $\underline{\text{Max}} \text{ F}_X$ can be stratified in smooth strata  in a natural manner. In particular the closed stratum 
 will provide a smooth center $Y(\subset  \underline{\text{Max}} \text{ F}_X)$, and hence a blow up $X=X_0\leftarrow X_1$.
\item If the characteristic is zero, a finite sequence of blow ups 
at smooth centers is constructed 
\begin{equation}
\label{bajada_maximo}
X=X_0 \longleftarrow X_1 \longleftarrow \ldots \longleftarrow X_s,
\end{equation}
by iteration of (1), so that 
\begin{equation}
\label{maximo_baja}
{\text{Max}} \text{ F}_X={\text{Max}} \text{ F}_{X_0}={\text{Max}} \text{ F}_{X_1}=\ldots >{\text{ Max}} \text{ F}_{X_s}.
\end{equation} 
\item (Equivariance) Assume that $F$ is equivariant, and let $\Theta: X'\to X$ be an isomorphism (so for all $\xi\in X'$, $\text{ F}_{X'}(\xi)=\text{ F}_{X}(\Theta(\xi))$).   
Then the smooth stratification of $\underline{\text{Max}}\text{ F}_{X'}$ from  (1) is that induced  by the smooth 
stratification of $\underline{\text{Max}}\text{ F}_X$ via pull back. Moreover,  if the characteristic is zero,  the sequence 
\begin{equation}
\label{b_m_p}
X'=X'_0 \longleftarrow X'_1 \longleftarrow \ldots \longleftarrow X'_{s'}
\end{equation}
with 
\begin{equation}
\label{maximo_baja_bis}
{\text{Max}} \text{ F}_{X'_0}={\text{Max}} \text{ F}_{X'_1}=\ldots >{\text{ Max}} \text{ F}_{X'_{s'}}, 
\end{equation}
 in (2), coincides with that induced by the sequence (\ref{bajada_maximo})  via pull backs with $\Theta: X'\to X$. Thus, in particular, $s'=s$. 
 \end{enumerate}
\end{theorem}

\begin{corollary}
\label{primero}
Let $X$ be a non-smooth variety over a perfect field. Then both, the  maximum stratum of the Hilbert-Samuel Function of
$X$, $\underline{\text{Max}} \ \text{HS}_X$,  and the maximum stratum of the multiplicity of $X$, $\underline{\text{Max}} \ \text{mult}_X$,  can be   stratified (in a natural  way). When the characteristic of the base field is zero, the maximum value of any of those  functions 
can be lowered via a finite sequence of blow up at smooth centers. Moreover, the process is constructive and equivariant.  
\end{corollary}

\begin{proof} Both, the Hilbert-Samuel Function and the Multiplicity of $X$ are representable via local embedding 
(see Examples \ref{ejemplo_Aroca}  and \ref{pres_mult}), and equivariant. So the claim follow from Theorem \ref{estratificar_X}.
\end{proof} 
\begin{parrafo}{\em On the proof of Theorem \ref{estratificar_X}}
%

To fix ideas let us consider the case $F=\text{ HS}$, the Hilbert Samuel function. As it has been indicated above, if  $\Theta: X'\to X$ is an isomorphism of scheme, then for all $\xi \in X'$, 
$\text{ HS}_{X'}(\xi)=\text{ HS}_X(\Theta(\xi))$. The claim, in this case, is that the procedure of reduction of $\max HS_X$  can be done in a way that also inherits the property of compatibility with isomorphisms: More precisely, a refinement of the function $F_X$ will lead to a new stratification into locally closed sets, and the task is to obtain such refinement so that:
\begin{itemize}
\item The strata of this new stratification are regular.
\item By successive blow ups at the closed and smooth center produced by this new stratification one achieves the reduction in (\ref{maximo_baja}).
\item This refined function should inherit the nice natural properties of the original function $F_X$, such as the compatibility with arbitrary isomorphisms of schemes.
\end{itemize}
\color{blue}

\color{black}
In order to achieve these results over fields of characteristic zero our aim is to use properties of {\em constructive resolution of basic objects}. Note here that basic objects are obtained only locally (see \ref{d_representable1}). In addition we want to guarantee the condition of equivariance in (3). We outline here the overall strategy, and we refer to 
\cite{Cut2}, or \cite{BrV1}, among other introductory presentations, for full details.
Assume first that the variety $X$ is (globally) embedded in a smooth scheme $V^{(n)}$ of dimension $n$
$$X\subset V^{(n)}$$ 
and that there is  an ${\mathcal O}_{V^{(n)}}$-Rees algebra ${\mathcal G}^{(n)}$ in the conditions of Definition \ref{d_representable}.
Recall that a resolution
\begin{equation}\label{svlap}
\begin{array}{ccccccc} 
V^{(n)}_0=V^{(n)} & \leftarrow & V^{(n)}_1 & \leftarrow & \ldots & \leftarrow V^{(n)}_m\\
{\mathcal G}_0^{(n)}= {\mathcal G}^{(n)} & & {\mathcal G}_1^{(n)} & &\ldots  & {\mathcal G}_m^{(n)}
\end{array}
\end{equation} 
induces a sequence  of blow ups over $X$, say
$$\begin{array}{ccccccc}
X=X_0 &  \leftarrow &  X_1 &  \leftarrow  &  \ldots & \leftarrow  & X_m, \mbox{ and }
\end{array}$$
$$\text{max F}_X=\text{max F}_{X_0}\dots = \text{max F}_{X_{m-1}} > \text{max F}_{X_m} $$

One can {\em construct} the resolution (\ref{svlap}). More precisely, given any pair $(V^{(n)}\mathcal G^{(n)})$, constructive resolution of basic objects provides a suitable upper-semi-continuous function on $\Sing(\mathcal G^{(n)})$.
The points where this function takes the biggest value is a smooth center in $\Sing(\mathcal G^{(n)})$, and 
(\ref{svlap}) is constructed by blowing up, successively, along such centers.

We will not discuss here how these upper-semi-continuos functions are defined, but we will give some evidence on why 
this procedure of resolution of pairs will lead us to the proof of the theorem.

To this end let us indicate that these upper semi-continuos functions have a main ancestor, namely the function
\begin{equation}\label{nge1}\ord^{(n)}: \mbox{Sing }{\mathcal G^{(n)}}\to \mathbb{Q}
\end{equation} 
in \ref{orderRees}. One readily checks that it  is upper semi-continuous and takes only finitely many values. On the other hand $\mbox{Sing }{\mathcal G_0^{(n)}}=\Max F_X$, so (\ref{nge1}) induces
\begin{equation}\label{nge2}\Max F_X\to \mathbb{Q}
\end{equation}
%
%

{\em Main Result I)} Assume that the same variety $X$ can be embedded in a smooth scheme $V'^{(n)}$, of dimension $n$ as before, and there is  an ${\mathcal O}_{V'^{(n)}}$-Rees algebra ${\mathcal G'}^{(n)}$ in the conditions of Definition \ref{d_representable}, then the induced function $\Max F_X\to \mathbb{Q}$ coincides with that in (\ref{nge2}). In other words, the same closed set $\Max F_X$ can be identified with $\Sing({\mathcal G}^{(n)}) (\subset V^{(n)})$ and with $\Sing({\mathcal G'}^{(n)}) (\subset V'^{(n)})$. The result says that a function, say
\begin{equation}\label{nge3}ord_X^{(n)}:\Max F_X\to \mathbb{Q}
\end{equation}
is well defined.

This first Main Result is due to Hironaka. The proof he gives is known as Hironaka's trick (see \cite{Cut2}, Proposition 6.27) , or \cite{EncVil97:Tirol}). Moreover, this proof goes one step
beyond:  As was indicated in the discussion in \ref{about_equivariance}, if 
$\Theta: X' \to X$ is an isomorphism of schemes and if $X\subset V^{(n)}$ and ${\mathcal G}^{(n)}$ are in the conditions of Definition \ref{d_representable}, then the isomorphism induce an inclusion $X' \subset V^{(n)}$, and in addition, and ${\mathcal G}^{(n)}$ is also in the conditions of Definition \ref{d_representable} for $X'$. 

{\em Main Result II)} Let $\Theta: X' \to X$ be an isomorphism of schemes, then this isomorphism maps $\Max F_{X'}$ to 
$\Max F_X$, and the functions 
$\ord_Y^{(n)}:\Max F_Y\to \mathbb{Q}$ and 
$\ord_X^{(n)}:\Max F_X\to \mathbb{Q}$, both obtained as in (\ref{nge2}), are compatible with the isomorphism. Namely, 
$$\text{ord}^{(n)}_{X'}(\xi)=\text{ord}^{(n)}_{X}(\Theta(\xi))$$
for all $\xi\in X'$.

As was previously mentioned, the functions defining the constructive resolution of $(V^{(n)}, {\mathcal G}^{(n)}, E^{(n)}=\{\emptyset\})$, used here in (\ref{svlap}), evolve from the functions in  (\ref{nge1}), at least 
for the procedure in \cite{Villa89} and \cite{Villa92}, where constructive resolution uses
$F_X=HS_X$ in \ref{ejemplo_Aroca}.


There are also alternative proofs of these Main Results which do not make use of Hironaka's trick. Recall from 
Remark \ref{lo_q}, 2), that if we fix an embedding of $X$, say 
$X \subset V^{(n)}$. the role of $(V^{(n)}, {\mathcal G}^{(n)}, E^{(n)}=\{\emptyset\})$ is well defined up to weak equivalence (see Def \ref{Defweakx}).  
In \cite{BrV1} it is proved firstly that the function in (\ref{nge1}) is well defined up to weak equivalence, a result that we have stated here in Theorem \ref{Th416}. Then, the notion of equivalence 
is extended so that if $X\subset V^{(n)}$ , ${\mathcal G}^{(n)}$, and $X\subset V^{(m)}$, ${\mathcal G'}^{(m)}$ are both in the conditions of Definition \ref{d_representable}, the pairs 
$(V^{(n)}, {\mathcal G}^{(n)})$, $(V^{(m)}, {\mathcal G'}^{(m)})$, are equivalent. Finally it is proved that in the preceding conditions both induce the same function (\ref{nge2}), and furthermore, the 
constructive resolution of both pairs induce the same sequence of monoidal transformations over $X$, say
$$X=X_0 \longleftarrow X_1 \longleftarrow \ldots \longleftarrow X_r.$$
This shows that the constructive resolution of basic objects is defined so as to ensure the conditions of globalization. Compatibility with isomorphisms follows from the main Result II).

%
%
%
%
%
%
 
 This is the accomplishment of constructive resolution of basic objects, and leads to:

\begin{theorem} 
\label{resolucion_alg_varieties_1}
Let $X$ be a non-smooth  variety  over a field of characteristic zero. Then  a  finite sequence 
of blow ups at smooth equimultiple centers can be constructed: 
\begin{equation}
\label{resolver_X_mas}
X=X_0 \longleftarrow X_1 \longleftarrow \ldots \longleftarrow X_r
\end{equation}
so that:
\begin{enumerate}
\item  $X_r$ is smooth;
\item The composition $X\leftarrow X_r$ induces an isomorphism on $X\setminus \Sing X$;
\item  The   exceptional divisor of $X\leftarrow X_r$ has normal crossing support.  
\end{enumerate}
Moreover, the process is constructive and equivariant. 
\end{theorem}

\begin{proof} Use   Theorem  \ref{estratificar_X} to construct a resolution of   singularities  of $X$, 
\begin{equation}
\label{short_resolution}
X=X_0 \longleftarrow X_1 \longleftarrow \ldots \longleftarrow X_s. 
\end{equation}
namely a sequence so that $\max \text{Mult}_{X_s}=1$. Thus $X\stackrel{\pi}{\longleftarrow} X_s$ satisfies conditions (1) and (2). 

Since each step  $X_i\longleftarrow X_{i+1}$   in the sequence (\ref{short_resolution}) is the blow at a smooth center $Y_i\subset \Sing X_i$, we can attach to it a well defined invertible sheaf ${\mathcal I}(Y_i){\mathcal O}_{X_{i+1}}$.  
Now define  ${\mathcal K}$ as $ {\mathcal I}(Y_1)\cdots  {\mathcal I}(Y_{s-1}){\mathcal O}_{X_s}$. Observe that 
 ${\mathcal K}$ is a locally invertible ideal supported on  $\pi^{-1}(\Sing X)\subset X_r$. 

Set ${\mathcal G}_s$ as the Rees ring of the ideal ${\mathcal K}$, and $F_s:=\emptyset$. Consider the constructive  resolution of the  the basic object 
$(X_s, {\mathcal G}_s, F_s),$
say   
\begin{equation}
\label{segunda_resolucion}
(X_s,{\mathcal G}_s, F_s)\longleftarrow (X_{s+1} ,{\mathcal G}_{s+1}, F_{s+1}) \longleftarrow \ldots 
\longleftarrow (X_r,{\mathcal G}_r, F_r).
\end{equation}
Notice that:
\begin{enumerate}
\item Since sequence (\ref{segunda_resolucion}) is a composition of blow ups at smooth centers, $X_r$ is smooth. 
\item Since for $i=s,\ldots, r-1$ the center of each blow up 
$$(X_i,{\mathcal G}_i, F_i)\longleftarrow (X_{i+1} ,{\mathcal G}_{i+1}, F_{i+1})$$
is contained in $\Sing ({\mathcal G}_i)$,  
$X\longleftarrow X_s\longleftarrow X_r$ 
induces an isomorphism  on $X\setminus \Sing X$;
\item Since  $\Sing ({\mathcal G}_r)=\emptyset$, the total transform of ${\mathcal K}$ in $X_r$ is supported on 
$F_r$. Thus the exceptional divisor of $X\longleftarrow X_r$ has normal crossings support (see \ref{541}). 
\end{enumerate} 
Finally, to see that the process is equivariant, observe that if $\Theta: X'\to X$ is an isomorphism, then the equivariant resolution of $X$  given, as before,  induces the  resolution of $X'$, 
$$\xymatrix{
X=X_0  & \ar[l] X_1 & \ldots \ar[l] & X_s  \ar[l] \\
X'=X'_0 \ar[u]_{\theta=\theta_0} & \ar[l] X'_1 \ar[u]_{\theta_1} & \ldots \ar[l] &  \ar[l]  X'_s \ar[u]_{\theta_s}
}$$
The latter is therefore obtained by blowing up smooth centers $Y'_i\subset \Sing X'_i$ with $\theta_i(Y_i')=Y_i$. Thus, if ${\mathcal K}'={\mathcal I}(Y'_1)\cdots  {\mathcal I}(Y'_{s-1}){\mathcal O}_{X'_s}$, then $\theta_s$ induces an isomorphism between ${\mathcal K}$ and ${\mathcal K}'$. Therefore, the basic objects,  $(X_s,{\mathcal G}_s, F_s=\emptyset)$ and $(X'_s,{\mathcal G'}_s, F'_s=\emptyset)$ defined by the corresponding Rees rings are identifiable (see \cite{BrV1}), and the isomorphism identifies the constructive resolution of both.
\end{proof} 
\begin{remark} We have formulated resolution of singularities for 
a variety over a field $k$ of characteristic zero. The Definition \ref{d_representable1}, and Theorem \ref{estratificar_X}, still apply for schemes that are simply equidimensional. In this case the algorithm of resolution of basic objects produces a sequence of blow ups at equimultiple centers
\begin{equation}
\label{resolver_X_mas1}
X=X_0 \longleftarrow X_1 \longleftarrow \ldots \longleftarrow X_r
\end{equation}
so that the multiplicity is constant along each irreducible component of $X_r$. Moreover, if each $X_i$ is replaced by $(X_i)_{red}$ we get the resolution of $X_{red}$ assigned by Theorem \ref{resolucion_alg_varieties_1}.


\end{remark}

\end{parrafo}

\end{document}